\documentclass[draft]{birkjour}

\usepackage[noadjust]{cite}
\usepackage{xcolor}
\RequirePackage[all]{xy}

\usepackage{longtable}
\usepackage{lscape}
\usepackage{afterpage}

\newtheorem{thm}{Theorem}[section]

\newtheorem{lem}[thm]{Lemma}

\theoremstyle{definition}

\theoremstyle{remark}
\newtheorem{rem}[thm]{Remark}

\numberwithin{equation}{section}

\newcommand{\BibTeX}{B\kern-0.1emi\kern-0.017emb\kern-0.15em\TeX}
\newcommand{\XYpic}{$\mathrm{X\kern-0.3em\raisebox{-0.18em}{Y}}$-$\mathrm{pic}\,$}

\newcommand{\cl}{C \kern -0.1em \ell}  

\newcommand{\BR}{\mathbb{R}}
\newcommand{\BC}{\mathbb{C}}

\newcommand{\ed}{\end{document}}

\def\F{{\mathbb F}}

\def\BC{{\mathbb C}}

\def\Z{{\rm Z}}
\def\H{{\rm H}}

\def\Z{{\rm Z}}
\def\Aut{{\rm Aut}}
\def\ker{{\rm ker}}

\def\ad{{\rm ad}}
\def\mod{{\rm \;mod\; }}
\def\Lambd{{\rm\Lambda}}

\newcommand{\C}{C \kern -0.1em \ell} 

\begin{document}

%
%
%
%
%
%
%
%
%

\title[A Note on Centralizers and Twisted Centralizers in Clifford Algebras]
 {A Note on Centralizers and Twisted Centralizers in Clifford Algebras}
\author[E. Filimoshina]{Ekaterina Filimoshina}
\address{%
HSE University\\
Moscow 101000\\
Russia}
\email{filimoshinaek@gmail.com}
%
\author[D. Shirokov]{Dmitry Shirokov}
\address{%
HSE University\\
Moscow 101000\\
Russia
\medskip}
\address{
and
\medskip}
\address{
Institute for Information Transmission Problems of the Russian Academy of Sciences \\
Moscow 127051 \\
Russia}
\email{dm.shirokov@gmail.com}
\date{\today}
\dedicatory{Last Revised:\\ \today}
\begin{abstract}
This paper investigates centralizers and twisted centralizers in degenerate and non-degenerate Clifford (geometric) algebras. We provide an explicit form of the centralizers and twisted centralizers of the subspaces of fixed grades, subspaces determined by the grade involution and the reversion, and their direct sums. The results can be useful for applications of Clifford algebras in computer science, physics, and engineering.
\end{abstract}
\label{page:firstblob}
\maketitle

\section{Introduction}

In this work, we consider degenerate and non-degenerate real and complex Clifford (geometric) algebras $\C_{p,q,r}$ of arbitrary dimension and signature  (in
the case of any complex Clifford algebra, we can take $q = 0$). 
Degenerate Clifford algebras have  applications in physics \cite{br1}, geometry \cite{pga1,gunn1,ma}, computer vision and image processing \cite{cv1}, motion capture and robotics \cite{ro2}, neural networks and machine learning \cite{tf,cNN0,cNN}, etc.

Several recent works on Clifford algebras use the notion of centralizers and twisted centralizers in $\C_{p,q,r}$  \cite{cNN,garling,HelmBook,OnInner,GenSpin,OnSomeLie,ICACGA}.
We call \textit{a centralizer of a set in $\C_{p,q,r}$} a subset of all elements of $\C_{p,q,r}$ that commute with all elements of this set. 
\textit{A twisted centralizer of a set in $\C_{p,q,r}$} is a subset of such multivectors that their projections onto the even $\C^{(0)}_{p,q,r}$ and odd $\C^{(1)}_{p,q,r}$ subspaces commute and anticommute respectively with all elements of this set (see details in Section \ref{section_centralizers}). 
Centralizers and twisted centralizers of some particular sets in $\C_{p,q,r}$ are used in literature for various purposes. 
For example, the recent paper \cite{cNN} finds an explicit form of the twisted centralizer of the grade-$1$ subspace in $\C_{p,q,r}$ and applies it in the construction of Clifford group equivariant neural networks. 
The work \cite{br1} uses the explicit form of the same twisted centralizer when considering degenerate spin groups. 
The book \cite{garling} finds an explicit form of the centralizer of the even subspace in the case of the non-degenerate Clifford algebra $\C_{p,q,0}$.
In the papers \cite{OnInner,GenSpin}, the centralizers and twisted centralizers of the even subspace and the grade-$1$ subspace are employed in consideration of Lie groups preserving the even and odd subspaces under the adjoint and twisted adjoint representations in the non-degenerate Clifford algebras $\C_{p,q,0}$.
The works \cite{OnSomeLie,ICACGA} find an explicit form of these centralizers in the case of arbitrary $\C_{p,q,r}$.

In light of appearance of centralizers and twisted centralizers in $\C_{p,q,r}$ in the recent papers,
we decided to investigate them in this note.
We concentrate on the centralizers and twisted centralizers of the subspaces of fixed grades $\C^{k}_{p,q,r}$, $k=0,1,\ldots,n$, the subspaces $\C^{\overline{m}}_{p,q,r}$, $m=0,1,2,3$, determined by the grade involution and the reversion, and their direct sums. In particular, we consider the centralizers and twisted centralizers of the even $\C^{(0)}_{p,q,r}$ and odd $\C^{(1)}_{p,q,r}$ subspaces. We find an explicit form of these centralizers and twisted centralizers in the case of arbitrary $k=0,1,\ldots,n$ and $m=0,1,2,3$. 
We study the relations between the considered centralizers and the twisted centralizers.
This paper also considers the centralizers and twisted centralizers in the particular cases of the non-degenerate Clifford algebra $\C_{p,q,0}$ and the Grassmann algebra $\C_{0,0,n}$. Theorems \ref{theorem_cc_k_even}, \ref{centralizers_qt}, \ref{centralizers_qt_ds} and Lemmas \ref{zm_zm1}, \ref{zl_zm}, \ref{lemma_KLV} are new.

The paper is structured as follows. 
Section \ref{section_Cpqr} introduces all the necessary notation related to $\C_{p,q,r}$. 
Section \ref{section_centralizers} provides an explicit form of the centralizers and twisted centralizers of the subspaces $\C^{k}_{p,q,r}$, $k=0,1,\ldots,n$ and considers the relations between them.
In Section \ref{section_examples}, we write out all the considered centralizers and twisted centralizers in the particular cases $\C_{p,q,0}$ and $\C_{0,0,n}$ and in the case of small $k\leq4$.
Section \ref{section_centralizers_qt} provides an explicit form of the centralizers and twisted centralizers of the subspaces $\C^{\overline{m}}_{p,q,r}$, $m=0,1,2,3$, and their direct sums, in particular, the even and odd subspaces.
The conclusions follow in Section \ref{section_conclusions}.
In Appendix \ref{sec_tild}, we describe the results related to the twisted adjoint representation with different signs, which is important for geometrical applications.

\section{Degenerate and Non-degenerate Clifford Algebras $\C_{p,q,r}$}\label{section_Cpqr}

Let us consider the Clifford (geometric) algebra \cite{hestenes,lounesto,p} $\C(V)=\C_{p,q,r}$, $p+q+r=n\geq1$, over a vector space $V$ with a symmetric bilinear form, where $V$ can be real $\BR^{p,q,r}$ or complex  $\BC^{p+q,0,r}$. We use $\F$ to denote the field of real numbers $\BR$ in the first case and the field of complex numbers $\BC$ in the second case. In this work, we consider both the case of the non-degenerate Clifford algebras $\C_{p,q,0}$ and the case of the degenerate Clifford algebras $\C_{p,q,r}$, $r\neq0$.
We use $\Lambda_r$ to denote the subalgebra $\C_{0,0,r}$, which is the Grassmann (exterior) algebra \cite{phys,lounesto}.
The identity element is denoted by $e$, the generators are denoted by $e_a$, $a=1,\ldots,n$. The generators satisfy the following conditions:
\begin{eqnarray}
    e_a e_b + e_b e_a = 2 \eta_{ab}e,\qquad \forall a,b=1,\ldots,n,
\end{eqnarray}
where $\eta=(\eta_{ab})$ is the diagonal matrix with $p$ times $+1$, $q$ times $-1$, and $r$ times $0$ on the diagonal in the real case $\C(\BR^{p,q,r})$ and $p+q$ times $+1$ and $r$ times $0$ on the diagonal in the complex case $\C(\BC^{p+q,0,r})$.

Let us consider the subspaces $\C^k_{p,q,r}$ of fixed grades $k=0,\ldots,n$. Their elements are linear combinations of the basis elements $e_A=e_{a_1\ldots a_k}:=e_{a_1}\cdots e_{a_k}$, $a_1<\cdots<a_k$, labeled by ordered multi-indices $A$ of length $k$, where $0\leq k\leq n$. The multi-index with zero length $k=0$ corresponds to the identity element $e$. The grade-$0$ subspace is denoted by $\C^0$ without the lower indices $p,q,r$, since it does not depend on the Clifford algebra's signature. We have $\C^{k}_{p,q,r}=\{0\}$ for $k<0$ and $k>n$.
Let us use the following notation:
\begin{eqnarray}
\C^{\geq k}_{p,q,r}&:=&\C^{k}_{p,q,r}\oplus\C^{k+1}_{p,q,r}\oplus\cdots\oplus\C^{n}_{p,q,r},
\\
\C^{\leq k}_{p,q,r}&:=&\C^0\oplus\C^{1}_{p,q,r}\oplus\cdots\oplus\C^{k}_{p,q,r}
\end{eqnarray}
for $0\leq k\leq n$. For example, $\C^{\geq0}_{p,q,r}=\C^{\leq n}_{p,q,r}=\C_{p,q,r}$ and $\C^{\geq n}_{p,q,r}=\C^{n}_{p,q,r}$.

Consider such conjugation operations as grade involution and reversion. The grade involute of an element  $U\in\C_{p,q,r}$ is denoted by $\widehat{U}$ and the reversion is denoted by $\widetilde{U}$. These operations satisfy
\begin{eqnarray}\label{^uv=^u^v}
\widehat{UV}=\widehat{U}\widehat{V},\qquad \widetilde{UV}=\widetilde{V}\widetilde{U},\qquad \forall U, V\in \C_{p,q,r}.
\end{eqnarray}

The grade involution defines the even $\C^{(0)}_{p,q,r}$ and odd $\C^{(1)}_{p,q,r}$ subspaces:
\begin{eqnarray}
\C^{(k)}_{p,q,r} = \{U\in\C_{p,q,r}:\;\; \widehat{U}=(-1)^kU\}=\bigoplus_{j=k\mod{2}}\C^j_{p,q,r},\quad k=0,1.
\end{eqnarray}
We can represent any element $U\in\C_{p,q,r}$ as a sum 
\begin{eqnarray}
U=\langle U\rangle_{(0)}+\langle U\rangle_{(1)},\qquad \langle U\rangle_{(0)}\in\C^{(0)}_{p,q,r},\quad \langle U\rangle_{(1)}\in\C^{(1)}_{p,q,r}.
\end{eqnarray}
We use the angle brackets $\langle \cdot \rangle_{(l)}$ to denote the operation of projection of multivectors onto the subspaces $\C^{(l)}_{p,q,r}$, $l=0,1$.
For an arbitrary subset $\H\subseteq\C_{p,q,r}$, we have
\begin{eqnarray}
    \langle \H \rangle_{(0)} := \H\cap\C^{(0)}_{p,q,r},\qquad  \langle \H \rangle_{(1)} := \H\cap\C^{(1)}_{p,q,r}.\label{H_even}
\end{eqnarray}

The grade involution and the reversion define four subspaces $\C^{\overline{0}}_{p,q,r}$,  $\C^{\overline{1}}_{p,q,r}$,  $\C^{\overline{2}}_{p,q,r}$, and  $\C^{\overline{3}}_{p,q,r}$ (they are called the subspaces of quaternion types $0, 1, 2$, and $3$ respectively in the papers \cite{quat1, quat2, quat3}):
\begin{eqnarray}
\C^{\overline{k}}_{p,q,r}=\{U\in\C_{p,q,r}:\; \widehat{U}=(-1)^k U,\;\; \widetilde{U}=(-1)^{\frac{k(k-1)}{2}} U\},\; k=0, 1, 2, 3.\label{qtdef}
\end{eqnarray}
Note that the Clifford algebra $\C_{p,q,r}$ can be represented as a direct sum of the subspaces $\C^{\overline{k}}_{p,q,r}$, $k=0, 1, 2, 3$, and viewed as $\mathbb{Z}_2\times\mathbb{Z}_2$-graded algebra with respect to the commutator and anticommutator \cite{b_lect}.
To denote the direct sum of different subspaces, we use the upper multi-index and omit the direct sum sign. For instance, $\C^{(1)\overline{2}4}_{p,q,r}:=\C^{(1)}_{p,q,r}\oplus\C^{\overline 2}_{p,q,r}\oplus\C^{4}_{p,q,r}$.

\section{Centralizers and Twisted Centralizers of the Subspaces of Fixed Grades}\label{section_centralizers}

Consider
the subset  $\Z^{m}_{p,q,r}$ of all elements of $\C_{p,q,r}$ commuting with all elements of the grade-$m$ subspace for some fixed $m$:
\begin{eqnarray}
\Z^{m}_{p,q,r}:=\{X\in\C_{p,q,r}:\quad X V = V X,\quad \forall V\in\C^{m}_{p,q,r}\}.\label{def_Zm}
\end{eqnarray}
Note that  ${\Z}^{m}_{p,q,r}=\C_{p,q,r}$ for $m<0$ and $m>n$.
We call the subset $\Z^{m}_{p,q,r}$ the \textit{centralizer} (see, for example, \cite{alg,garling}) of the subspace $\C^m_{p,q,r}$ in $\C_{p,q,r}$.

The center of the Clifford algebra $\C_{p,q,r}$ is also the centralizer but of the entire Clifford algebra $\C_{p,q,r}$.
We denote the center of the degenerate and non-degenerate Clifford algebra $\C_{p,q,r}$ by $\Z_{p,q,r}$. It is well known (see, for example, \cite{br1}) that
\begin{eqnarray}\label{Zpqr}
\Z_{p,q,r}=
\left\lbrace
\begin{array}{lll}
\Lambd^{(0)}_{r}\oplus\C^n_{p,q,r},&&\mbox{$n$ is odd},
\\
\Lambd^{(0)}_{r},&&\mbox{$n$ is even}.
\end{array}
\right.
\end{eqnarray}

Similarly, consider 
the set $\check{\Z}^m_{p,q,r}$:
\begin{eqnarray}
    \check{\Z}^m_{p,q,r}:=\{X\in\C_{p,q,r}:\quad \widehat{X} V = V X,\quad \forall V\in\C^{m}_{p,q,r}\}.\label{def_chZm}
\end{eqnarray}
Note that $\check{{\Z}}^{m}_{p,q,r}=\C_{p,q,r}$ for $m<0$ and $m>n$.
We call the set $\check{{\Z}}^{m}_{p,q,r}$ the \textit{twisted centralizer} of the subspace $\C^m_{p,q,r}$ in $\C_{p,q,r}$. 
The particular case $\check{\Z}^{1}_{p,q,r}$ is considered in the papers \cite{cNN,br1,OnSomeLie}.

Note that twisted centralizers can be defined in another way, which we denote by $\tilde{\Z}^m_{p,q,r}$:
\begin{eqnarray}
    \tilde{\Z}^m_{p,q,r}:=\{X\in\C_{p,q,r}:\quad X\langle V\rangle_{(0)} + \widehat{X} \langle V\rangle_{(1)} = VX,\quad \forall V\in\C^{m}_{p,q,r}\}.
\end{eqnarray}
We write out an explicit form of these objects in Appendix \ref{sec_tild}.

Note that the projections $\langle\Z^m_{p,q,r}\rangle_{(0)}$ and $\langle\check{\Z}^m_{p,q,r}\rangle_{(0)}$ of $\Z^{m}_{p,q,r}$ and $\check{\Z}^m_{p,q,r}$ respectively onto the even subspace $\C^{(0)}_{p,q,r}$  (\ref{H_even}) coincide by definition:
\begin{eqnarray}
    \langle\Z^m_{p,q,r}\rangle_{(0)}=\langle\check{\Z}^m_{p,q,r}\rangle_{(0)},\qquad \forall m=0,1,\ldots,n.\label{same_even}
\end{eqnarray}
In the case $m=0$, we have
\begin{eqnarray}\label{chCC_0}
\Z^0_{p,q,r}=\C_{p,q,r},\qquad \check{\Z}^0_{p,q,r}=\{X\in\C_{p,q,r}:\quad \widehat{X}=X\}=\C^{(0)}_{p,q,r}.
\end{eqnarray}
In Theorem \ref{theorem_cc_k_even}, we find explicit forms of the centralizers $\Z^m_{p,q,r}$ and the twisted centralizers  $\check{\Z}^m_{p,q,r}$ of the subspaces of fixed grades for an arbitrary $m=1,\ldots,n$.

To prove Theorem \ref{theorem_cc_k_even}, let us prove auxiliary Lemmas \ref{zm_zm1}, \ref{zl_zm}, and \ref{lemma_KLV}.
In Lemmas \ref{zm_zm1} and \ref{zl_zm}, we use that any non-zero $X\in\C_{p,q,r}$ has the following decomposition over a basis:
\begin{eqnarray}\label{bas_dec}
X=X_1+\cdots+X_k,\quad X_i=\alpha_ie_{A_i},\quad  \alpha_i\in\F^{\times},\quad i=1,\ldots,k,
\end{eqnarray}
where $A_i$ is an ordered multi-index, $A_i\neq A_j$ for $i\neq j$, and each $X_i$ is non-zero, i.e. $\alpha_i\neq0$.

\begin{lem}\label{zm_zm1}
For any even $m$, we have
\begin{eqnarray}\label{st_lemma4.1}
\langle \check{\Z}^{m}_{p,q,r}\rangle_{(1)}\subseteq \langle \check{\Z}^{m+1}_{p,q,r}\rangle_{(1)}.
\end{eqnarray}
\end{lem}
\begin{proof}
In the case of even  $m<0$ or $m\geq n$, we have $\check{\Z}^m_{p,q,r}=\check{\Z}^{m+1}_{p,q,r}=\C_{p,q,r}$. For $m=0$, we have $\langle \check{\Z}^0_{p,q,r}\rangle_{(1)}=\langle \C^{(0)}_{p,q,r}\rangle_{(1)}=\{0\}$, where we use (\ref{chCC_0}). 
Let us consider the case $0<m<n$.
Consider a non-zero element $X\in\check{\Z}^{m}_{p,q,r}\cap\C^{(1)}_{p,q,r}$, where $m$ is even, and its 
 decomposition over a basis  (\ref{bas_dec}).
 For any fixed $e_{a_1\ldots a_{m}}\in\C^{m}_{p,q,r}$,  each summand $X_i$, $i=1,\ldots,k$, contains at least one such $e_{x_i}$ that $x_{i}\in\{a_1,\ldots, a_{m}\}$, because otherwise we have $X_i e_{a_1\ldots a_m}=e_{a_1\ldots a_m}X_i$, so 
$\widehat{X}e_{a_1\ldots a_m}\neq e_{a_1\ldots a_m}X$,
and we get a contradiction. 
Therefore, for any $e_{a_1\ldots a_{m+1}}\in\C^{m+1}_{p,q,r}$,
\begin{eqnarray}\label{f_lemma_2_0}
\widehat{X}e_{a_1\ldots a_{m+1}}=\pm \widehat{X_1}e_{a_1\ldots\check{x}_1\ldots a_{m+1}}e_{x_1}\pm\cdots\pm \widehat{X_k}e_{a_1\ldots\check{x}_k\ldots a_{m+1}}e_{x_k},
\end{eqnarray}
where $X_i$ contains $e_{x_i}$, $i=1,\ldots,k$, the sign  depends on the parity of the corresponding permutation, and the checkmarks indicate that the corresponding indices are missing. From (\ref{f_lemma_2_0}), we obtain
\begin{eqnarray}
\widehat{X}e_{a_1\ldots a_{m+1}}=
\pm e_{a_1\ldots\check{x}_1\ldots a_{m+1}}X_1 e_{x_1}\pm\cdots\pm e_{a_1\ldots\check{x}_k\ldots a_{m+1}}X_k e_{x_k}\label{f_1}
\\
=\pm e_{a_1\ldots\check{x}_1\ldots a_{m+1}}e_{x_1} X_1\pm\cdots\pm e_{a_1\ldots\check{x}_k\ldots a_{m+1}}e_{x_k}X_k=e_{a_1\ldots a_{m+1}} X,\label{f_2}
\end{eqnarray}
where all the signs preceding the terms remain the same in (\ref{f_lemma_2_0})--(\ref{f_2}),
since $X_i\in\check{\Z}^{m}_{p,q,r}\cap\C^{(1)}_{p,q,r}$ and $X_i$ contains $e_{x_i}$, $i=1,\ldots,k$. This completes the proof.
\end{proof}

\begin{lem}\label{zl_zm}
For any odd $m$, we have
\begin{eqnarray}
    \langle\Z^m_{p,q,r}\rangle_{(1)}\subseteq\langle\Z^{m+1}_{p,q,r}\rangle_{(1)}.\label{formula_lemma4.2}
\end{eqnarray}
\end{lem}
\begin{proof}
In the case $m<0$ or $m\geq n$, we have ${\Z}^m_{p,q,r}={\Z}^{m+1}_{p,q,r}=\C_{p,q,r}$.
Let us consider the case $0<m<n$.
Consider a non-zero $X\in\Z^m_{p,q,r}\cap\C^{(1)}_{p,q,r}$, where $m$ is odd, and its decomposition (\ref{bas_dec}).   
For any fixed $e_{a_1\ldots a_{m}}\in\C^{m}_{p,q,r}$, each summand $X_i$, $i=1,\ldots,k$, contains at least one such $e_{x_i}$ that $x_{i}\in\{a_1,\ldots a_{m}\}$, because otherwise $X_i e_{a_1\ldots a_m}=e_{a_1\ldots a_m}\widehat{X_i}$, so 
$X e_{a_1\ldots a_m}\neq e_{a_1\ldots a_m} X$,
and we get a contradiction. 
Hence, for any $e_{a_1\ldots a_{m+1}}\in\C^{m+1}_{p,q,r}$, we have
\begin{eqnarray}\label{f_lemma_2_0_1}
X e_{a_1\ldots a_{m+1}}=\pm {X_1}e_{a_1\ldots\check{x}_1\ldots a_{m+1}}e_{x_1}\pm\cdots\pm{X_k}e_{a_1\ldots\check{x}_k\ldots a_{m+1}}e_{x_k},
\end{eqnarray}
where $X_i$ contains $e_{x_i}$, $i=1,\ldots,k$, the sign  depends on the parity of the corresponding permutation, and the checkmarks indicate that the corresponding indices are missing. From (\ref{f_lemma_2_0_1}), we obtain
\begin{eqnarray}
Xe_{a_1\ldots a_{m+1}}=
\pm e_{a_1\ldots\check{x}_1\ldots a_{m+1}}X_1 e_{x_1}\pm\cdots\pm e_{a_1\ldots\check{x}_k\ldots a_{m+1}}X_k e_{x_k}
\\
=\pm e_{a_1\ldots\check{x}_1\ldots a_{m+1}}e_{x_1} X_1\pm\cdots\pm e_{a_1\ldots\check{x}_k\ldots a_{m+1}}e_{x_k}X_k=e_{a_1\ldots a_{m+1}} X,\label{f_2_2}
\end{eqnarray}
where all the signs preceding the terms in (\ref{f_lemma_2_0_1})--(\ref{f_2_2}) remain the same,
since $X_i\in\Z^{m}_{p,q,r}\cap \C^{(1)}_{p,q,r}$
and $X_i$ contains $e_{x_i}$, $i=1,\ldots,k$. This completes the proof.
\end{proof}

\begin{rem}\label{rem_zlzm}
Note that the more general statement than (\ref{formula_lemma4.2}) holds true:
\begin{eqnarray}
    \Z^m_{p,q,r}\subseteq\Z^{m+1}_{p,q,r},\qquad \mbox{$m$ is odd},
\end{eqnarray}
which follows from Theorem \ref{theorem_cc_k_even} below and is provided in the formula (\ref{CC3CC4_2_2}) of  Remark \ref{CC3CC4}. 
Note that the statement (\ref{st_lemma4.1}) can not be generalized in a similar way. For any even $m$, we have
\begin{eqnarray}
&\langle \check{\Z}^m_{p,q,r} \rangle_{(0)}
\subseteq\langle \check{\Z}^{m+1}_{p,q,r} \rangle_{(0)},\quad \mbox{$n$ is odd},
\\
&(\langle \check{\Z}^m_{p,q,r} \rangle_{(0)}\setminus \C^{n}_{p,q,r})\subseteq\langle \check{\Z}^{m+1}_{p,q,r} \rangle_{(0)},\quad \mbox{$n$ is even},
\end{eqnarray}
which follows from Theorem \ref{theorem_cc_k_even} below as well.
\end{rem}

\begin{lem}\label{lemma_KLV}
For any $M\in\C^{m}_{p,q,r}$, $K\in\C^{k}_{p,q,r}$,  and $L\in\Lambda^{n-m}_r$, we have
\begin{eqnarray*}
(KL)M=
\left\lbrace
    \begin{array}{lll}
    \!M(KL)&&\!\!\!\!\!\!\mbox{if $m,k$ are even; $m,k,n$ are odd};
    \\
    \!M\widehat{(KL)}&&\!\!\!\!\!\!\mbox{if $m$ is odd, $k$ is even; $m,n$ are even, $k$ is odd.}
     \end{array}
    \right.
\end{eqnarray*}
\end{lem}
\begin{proof}
     Suppose $m,k=0\mod{2}$. We have 
     \begin{eqnarray}\label{KLV_1}
         (KL)M=(LM)K=(ML)K=M(KL),
     \end{eqnarray}
     where we use that $LM\in\C^{n}_{p,q,r}$ commutes with any even element, $LM=ML$, and $LK=KL$, since $m$ and $k$ are even respectively. 
     If $m,k,n=1\mod{2}$, then we again have (\ref{KLV_1}), since $L$ is even and $LM\in\C^{n}_{p,q,r}\subset\Z_{p,q,r}$ is odd.

     Consider the case $m=1\mod{2}$ and $k=0\mod{2}$. 
     If $n$ is odd, then $L$ is even. We get (\ref{KLV_1}) again and obtain $(KL)M=M\widehat{(KL)}$, since both $K$ and $L$ are even.
     If $n$ is even, then $L$ is odd and
     \begin{eqnarray}
         (KL)M=(LM)K=(M\widehat{L})K=M(K\widehat{L})=M\widehat{(KL)}.
     \end{eqnarray}
     Finally, suppose $m,n=0\mod{2}$ and $k=1\mod{2}$. We obtain 
     \begin{eqnarray}
         (KL)M=(LM)\widehat{K}=(ML)\widehat{K}=M(\widehat{K}L)=M\widehat{(KL)},
     \end{eqnarray}
     since $L$ is even, $ML\in\C^{n}_{p,q,r}$ is even, and it anticommutes with all odd elements, including $K$.
\end{proof}

\begin{rem}\label{rem_zero}
    Note that 
    \begin{eqnarray}
        \Lambda^{k}_r \C^{m}_{p,q,r}\subseteq \C^{k+m}_{p,q,r},\quad k\geq 1;\qquad \Lambda^{0}_r \C^{m}_{p,q,r} =  \C^{m}_{p,q,r}.
    \end{eqnarray}
    Moreover, if at least one of $k$ and $m$ is even, then 
    \begin{eqnarray}
        X V = V X,\qquad \forall X\in \Lambda^{k}_r,\quad \forall V\in \C^{m}_{p,q,r}.
    \end{eqnarray}
    If both $k$ and $m$ are odd, then
    \begin{eqnarray}
        \widehat{X} V = V X,\qquad \forall X\in \Lambda^{k}_r,\quad \forall V\in \C^{m}_{p,q,r}.
    \end{eqnarray}
\end{rem}
We use Remark \ref{rem_zero} in the proof of Theorem \ref{theorem_cc_k_even}.
In Theorem \ref{theorem_cc_k_even}, we find the centralizers and twisted centralizers for any $m=1,\ldots,n$ in the case $r\neq n$. The case of the Grassmann algebra $\C_{0,0,n}$ is written out separately in Remark \ref{remark_00n}  for the sake of brevity in the theorem statement.

A few words about the notation in Theorem \ref{theorem_cc_k_even}. The spaces $\C^k_{p,q,0}$ and $\Lambd^k_r$, $k=0,\ldots,n$, are regarded as subspaces of $\C_{p,q,r}$. 
By $\{\C^k_{p,q,0}\Lambd^l_r\}$, we denote the subspace of $\C_{p,q,r}$ spanned by the elements of the form $ab$, where $a\in \C^k_{p,q,0}$ and $b\in\Lambd^l_r$. 

\begin{thm}\label{theorem_cc_k_even}
Consider the case $r\neq n$.
\begin{enumerate}
    \item 
   For an arbitrary even $m$, where $n\geq m\geq 2$, the centralizer has the form
    \begin{eqnarray}\label{ccm_even}
    \Z^m_{p,q,r}=\Lambda^{\leq n-m-1}_r\oplus \bigoplus_{k=1\mod{2}}^{m-3} \{\C^{k}_{p,q,0}\Lambda^{\geq n-(m-1)}_r\}\nonumber
    \\
    \oplus \bigoplus_{k=0\mod{2}}^{m-2}\{\C^{k}_{p,q,0}\Lambda^{\geq n-m}_r\} \oplus \C^{n}_{p,q,r},
    \end{eqnarray}
    and the twisted centralizer is equal to
    \begin{eqnarray}
        \!\!\!\!\!\!\!\!\!\!\!\!&&\check{\Z}^m_{p,q,r}\!\!=\!\!
        \langle \Lambda_r^{\leq n-m-1}\!\oplus\! \!\!\!\!\!\!\bigoplus_{k=1\mod{2}}^{m-3}\!\!\!\!\!\! \!\{\C^{k}_{p,q,0}\Lambda^{\geq n-(m-1)}_r\!\}\!\oplus\! \!\!\!\!\!\!\bigoplus_{k=0\mod{2}}^{m-2}\!\!\!\!\!\! \!\{\C^{k}_{p,q,0}\Lambda^{\geq n-m}_r\!\} \!\oplus\! \C^{n}_{p,q,r} \rangle_{(0)}\nonumber
        \\
        \!\!\!\!\!\!\!\!\!\!\!\!&&\qquad\qquad\qquad\oplus \langle
        \!\!\!\!\bigoplus_{k=0\mod{2}}^{m-2}\!\!\!\!\!\{\C^{k}_{p,q,0}\Lambda^{\geq n-(m-1)}_r\}
        \oplus \!\!\!\!\!\bigoplus_{k=1\mod{2}}^{m-1}\!\!\!\!\!\{\C^{k}_{p,q,0}\Lambda^{\geq n-m}_r\} \rangle_{(1)}.\label{ch_cc_m_even}
    \end{eqnarray}

    \item
    For an arbitrary odd $m$, where $n\geq m \geq 1$, we have
    \begin{eqnarray}
    \check{\Z}^m_{p,q,r}=\Lambda_r^{\leq n-m-1}\oplus\!\!\!\!\!\!\bigoplus_{k=1\mod{2}}^{m-2}\!\!\!\!\!\{\C^{k}_{p,q,0}\Lambda_r^{\geq n-(m-1)}\}\oplus\!\!\!\!\!\!\bigoplus_{k=0\mod{2}}^{m-1}\!\!\!\!\! \{\C^k_{p,q,0}\Lambda_r^{\geq n-m}\}\label{ch_cc_odd}
    \end{eqnarray}
    and
     \begin{eqnarray}
        \!\!\!\!\!\!\!\!\!\!\!\!\!\!&&\Z^m_{p,q,r} = \langle \Lambda_r^{\leq n-m-1}\oplus \!\!\!\!\bigoplus_{k=1\mod{2}}^{m-2}\!\!\!\!\{\C^{k}_{p,q,0}\Lambda^{\geq n-(m-1)}_r\} \oplus \!\!\!\!\bigoplus_{k=0\mod{2}}^{m-1}\!\!\!\!\{\C^{k}_{p,q,0}\Lambda^{\geq n-m}_r\} \rangle_{(0)} \nonumber
        \\
        \!\!\!\!\!\!\!\!\!\!\!\!\!\!\!\!\!\!\!\!\!\!\!\!\!\!&&\;\;\quad\quad\oplus \langle \!\!\!\!\!\bigoplus_{k=0\mod{2}}^{m-3}\!\!\!\!\!\{\C^{k}_{p,q,0}\Lambda^{\geq n-(m-1)}_r\}\oplus \!\!\!\!\!\!\bigoplus_{k=1\mod{2}}^{m-2}\!\!\!\!\!\!\{\C^k_{p,q,0}\Lambda^{\geq n-m}_r\}\oplus\C^{n}_{p,q,r}\rangle_{(1)}.\label{cc_odd}
     \end{eqnarray}
\end{enumerate}
\end{thm}

\begin{proof}

Let us prove (\ref{ccm_even}). 
Namely, we prove that for any $X\in\C_{p,q,r}$ and even $m$, where $n\geq m\geq 2$, the condition $X e_{a_1\ldots a_m} = e_{a_1\ldots a_m} X$ for any basis element $e_{a_1\ldots a_m}\in\C^m_{p,q,r}$ is equivalent to the condition 
\begin{eqnarray*}
    X\!\in \Lambda^{\leq n-m-1}_r\oplus\! \!\!\!\!\!\!\bigoplus_{k=1\mod{2}}^{m-3}\!\!\!\!\!\! \{\C^{k}_{p,q,0}\Lambda^{\geq n-(m-1)}_r\}\oplus\!\!\!\!\!\!\!\bigoplus_{k=0\mod{2}}^{m-2}\!\!\!\!\!\!\{\C^{k}_{p,q,0}\Lambda^{\geq n-m}_r\} \!\oplus\! \C^{n}_{p,q,r}.
\end{eqnarray*}
For any fixed $a_1,\ldots, a_m$, we can always represent $X$ as a sum of $2^m$ summands:
\begin{eqnarray}
    \!\!\!\!\!\!X=Y  + e_{a_1}Y_{a_1}\! + \cdots + e_{a_m}Y_{a_m}\!\! + e_{a_1 a_2} Y_{a_1 a_2}\! + \cdots + e_{a_1\ldots a_m}Y_{a_1\ldots a_m},\label{X_form}
\end{eqnarray}
where $Y$, $Y_{a_1}$, \ldots, $Y_{a_1\ldots a_m}\in\C_{p,q,r}$ do not contain $e_{a_1},\ldots, e_{a_m}$. We get
\begin{eqnarray*}
    \!\!\!\!\!\!\!\!\!\!\!\!&&X e_{a_1\ldots a_m}=(Y + \cdots + e_{a_1\ldots a_m}Y_{a_1\ldots a_m}) e_{a_1\ldots a_m}
    \\
   \!\!\!\!\!\!\!\!\!\!\!\!&&= e_{a_1\ldots a_m}( \sum^{m}_{\substack{k=0\mod{2}}} e_{a_{i_1}\ldots a_{i_k}}Y_{a_{i_1}\ldots a_{i_k}}
    -\sum^{m-1}_{k=1\mod{2}} e_{a_{i_1}\ldots a_{i_k}}Y_{a_{i_1}\ldots a_{i_k}}),
\end{eqnarray*}
where $a_{i_1},\ldots,a_{i_k}\in\{a_1,\ldots,a_m\}$,
$a_{i_1}<\cdots<a_{i_k}$, the elements 
$e_A$ and $Y_A$ with the multi-indices of zero length are the identity element $e$ and $Y$ respectively, and the minus sign precedes summands with $e_{a_{i_1}\ldots a_{i_k}}\in\C^{(1)}_{p,q,r}$. 
We get that the condition $X e_{a_1\ldots a_m}=e_{a_1\ldots a_m}X$ is equivalent to
\begin{eqnarray}\label{eq_cc_p_1}
2e_{a_1\ldots a_m}(\sum^{m-1}_{k=1\mod{2}} e_{a_{i_1}\ldots a_{i_k}}Y_{a_{i_1}\ldots a_{i_k}})=0.
\end{eqnarray}
The equation (\ref{eq_cc_p_1}) is equivalent to the system of $2^{m-1}$ equations:
\begin{eqnarray}
     \!\!\!\!\!\!\!\!\!\!\!\!&(e_{a_1})^2 e_{a_2\ldots a_m}Y_{a_1}=0,\quad \ldots, \quad (e_{a_m})^2 e_{a_1\ldots a_{m-1}}Y_{a_m}=0,\label{eq_cc_p_2}
    \\
     \!\!\!\!\!\!&(e_{a_1})^2 (e_{a_2})^2 (e_{a_3})^2 e_{a_4\ldots a_m}Y_{a_1 a_2 a_3}\!=0,\;\ldots,\;(e_{a_2})^2\cdots(e_{a_m})^2e_{a_1}Y_{a_2\ldots a_m}\!=0.\nonumber\label{eq_cc_p_2_2}
\end{eqnarray}
Using $(e_{a_1})^2 e_{a_2\ldots a_m}Y_{a_1}=0$ (\ref{eq_cc_p_2}), we get that if $(e_{a_1})^2\neq0$, i.e. $a_1\in\{1,\ldots, p+q\}$, then $Y_{a_1}=0$, since $Y_{a_1}$ does not contain $e_{a_2},\ldots,e_{a_m}$. On the other hand, if each summand of $X$  either contains only the non-invertible generators, or contains  at least $1$ invertible generator and at the same time does not contain less than $m-1$ generators, then the equation $(e_{a_1})^2 e_{a_2\ldots a_m}Y_{a_1}=0$ is satisfied. Therefore, $(e_{a_1})^2 e_{a_2\ldots a_m}Y_{a_1}=0$ is satisfied if and only if  $X$ has no summands containing at least $1$ invertible generator and at the same time not containing $m-1$ or more of any generators. 
Similarly, for any other odd $k<m$, using $(e_{a_1})^2 (e_{a_2})^2 \ldots (e_{a_k})^2 e_{a_{k+1}\ldots a_m}Y_{a_1 a_2 \ldots a_k}=0$ (\ref{eq_cc_p_2}), we get $Y_{a_1 a_2 \ldots a_k}=0$ if $a_1,\ldots, a_k\in\{1,\ldots,p+q\}$. 
Moreover, the equation $(e_{a_1})^2 (e_{a_2})^2 \ldots (e_{a_k})^2 e_{a_{k+1}\ldots a_m}Y_{a_1 a_2 \ldots a_k}=0$ is satisfied if and only if $X$ has no summands containing at least $k$ invertible generators and at the same time not containing $m-k$ or more of any generators for any odd $k\leq m$. 
This implies that for 
\begin{eqnarray}\label{Xis}
X\in\C_{p,q,r}=\Lambda_r\oplus \{\C^{1}_{p,q,0}\Lambda_r\}\oplus\cdots \oplus\{\C^{p+q}_{p,q,0}\Lambda_r\},
\end{eqnarray}
we finally obtain
\begin{eqnarray}
    X\!\!\!\!&\in&\!\!\!\!\Lambda_r\oplus \{\C^{1}_{p,q,0}\Lambda^{\geq n-(m-1)}_r\}\oplus \{\C^2_{p,q,0}\Lambda^{\geq n-m}_r\}\oplus \{\C^3_{p,q,0}\Lambda^{\geq n-(m-1)}_r\} \nonumber
    \\
    &&\oplus \cdots\oplus \{\C^{m-3}_{p,q,0}\Lambda^{\geq n-(m-1)}_r\}\oplus \{\C^{m-2}_{p,q,0}\Lambda^{\geq n-m}_r\}\oplus \C^{n}_{p,q,r},
\end{eqnarray}
since for any fixed odd $k$ and even $k+1$, we have the following condition on $d$ for the subspaces $\{\C^{k}_{p,q,0}\Lambda^{d}_r\}$ and $\{\C^{k+1}_{p,q,0}\Lambda^{d}_r\}$ respectively: the number of not contained generators should be less than $m-k$, i.e. $n-(k+d)<m-k$, so $d\geq n-(m-1)$ for $\{\C^{k}_{p,q,0}\Lambda^{d}_r\}$, and $n-(k+1+d)<m-k$, thus, $d\geq n-m$ for $\{\C^{k+1}_{p,q,0}\Lambda^{d}_r\}$.
This completes the proof.

Let us prove (\ref{ch_cc_odd}). Namely, let us prove that for any $X\in\C_{p,q,r}$ and odd $m$, where $n\geq m\geq 1$, the condition  $\widehat{X}e_{a_1\ldots a_m}=e_{a_1 \ldots a_m}X$ for any basis element $e_{a_1 \ldots a_m}\in\C^{m}_{p,q,r}$ is equivalent to the condition
\begin{eqnarray}
    X\in\Lambda_r^{\leq n-m-1}\oplus\!\!\!\!\!\!\bigoplus_{k=1\mod{2}}^{m-2}\!\{\C^{k}_{p,q,0}\Lambda_r^{\geq n-(m-1)}\}\oplus\!\!\!\!\!\!\bigoplus_{k=0\mod{2}}^{m-1} \!\{\C^k_{p,q,0}\Lambda_r^{\geq n-m}\}.\label{ch_cc_odd_}
    \end{eqnarray}
For any fixed $a_1,\ldots,a_m$, we can represent $X$ as a sum of $2^m$ summands (\ref{X_form}), where $Y,\ldots,Y_{a_1\ldots a_m}\in\C_{p,q,r}$ do not contain $e_{a_1},\ldots,e_{a_m}$.
We obtain
\begin{eqnarray*}
    \!\!\!\!\!\!\!&&\widehat{X}e_{a_1\ldots a_m}=(\langle X\rangle_{(0)}-\langle X\rangle_{(1)})e_{a_1\ldots a_m}
    \\
    \!\!\!\!\!\!\!&&= e_{a_1\ldots a_m}(\sum^{m-1}_{k=0\mod{2}} e_{a_{i_1}\ldots a_{i_k}}Y_{a_{i_1}\ldots a_{i_k}}
    -\sum^{m}_{k=1\mod{2}} e_{a_{i_1}\ldots a_{i_k}}Y_{a_{i_1}\ldots a_{i_k}}),
\end{eqnarray*}
where  $a_{i_1},\ldots,a_{i_k}\in\{a_1,\ldots,a_m\}$, $a_{i_1}<\cdots<a_{i_k}$, the elements 
$e_A$ and $Y_A$ with the multi-indices of zero length are the identity element $e$ and $Y$ respectively, and the minus sign precedes summands with $e_{a_{i_1}\ldots a_{i_k}}\in\C^{(1)}_{p,q,r}$. 
We get that the equality $\widehat{X}e_{a_1\ldots a_m}=e_{a_1\ldots a_m}X$ is equivalent to the formula
\begin{eqnarray}\label{eq_cc_p_1_2}
2e_{a_1\ldots a_m}(\sum^{m}_{k=1\mod{2}} e_{a_{i_1}\ldots a_{i_k}}Y_{a_{i_1}\ldots a_{i_k}})=0.
\end{eqnarray}
Similar to how it is done for the formula (\ref{eq_cc_p_1}) above, from the formula (\ref{eq_cc_p_1_2}), we get that it
 is equivalent to the condition that $X$ has no summands containing at least $k$ invertible generators and at the same time not containing $m-k$ or more of any generators for any odd $k\leq m$. So, for $X\in\Lambda_r\oplus \{\C^{1}_{p,q,0}\Lambda_r\}\oplus\cdots \oplus\{\C^{p+q}_{p,q,0}\Lambda_r\}$, we get
\begin{eqnarray}
    X\!\!\!\!&\in&\!\!\!\!\Lambda_r\oplus \{\C^{1}_{p,q,0}\Lambda^{\geq n-(m-1)}_r\}\oplus \{\C^2_{p,q,0}\Lambda^{\geq n-m}_r\}\oplus \{\C^3_{p,q,0}\Lambda^{\geq n-(m-1)}_r\} \nonumber
    \\
    &&\oplus \cdots\oplus \{\C^{m-2}_{p,q,0}\Lambda^{\geq n-(m-1)}_r\}\oplus \{\C^{m-1}_{p,q,0}\Lambda^{\geq n-m}_r\},
\end{eqnarray}
since, similarly to the proof of (\ref{ccm_even}) above, for any odd $k$, for $\{\C^{k}_{p,q,0}\Lambda^{d}_r\}$, we have the condition $n-(k+d)<m-k$, i.e. $d\geq n-(m-1)$, and for any $\{\C^{k+1}_{p,q,0}\Lambda^{d}_r\}$, we get $n-(k+1+d)<m-k$, thus, $d\geq n-m$. This completes the proof.

Now we prove (\ref{ch_cc_m_even}). Suppose $m$ is even and $n\geq m\geq 2$. Since $\langle\check{\Z}^m_{p,q,r}\rangle_{(0)}=\langle \Z^m_{p,q,r}\rangle_{(0)}$ (\ref{same_even}), we only need to prove 
\begin{eqnarray}
    \!\!\langle\check{\Z}^m_{p,q,r}\rangle_{(1)}=\langle\!\!\!
        \bigoplus_{k=0\mod{2}}^{m-2}\!\!\!\{\C^{k}_{p,q,0}\Lambda^{\geq n-(m-1)}_r\}
        \oplus \!\!\!\!\bigoplus_{k=1\mod{2}}^{m-1}\!\!\!\{\C^{k}_{p,q,0}\Lambda^{\geq n-m}_r\} \rangle_{(1)}.\label{need_pr_1}
\end{eqnarray}
First, we prove that the right set is a subset of the left one in (\ref{need_pr_1}). 
We have $\{\C^{k}_{p,q,0}\Lambda^{n-m}_r\}\subseteq\check{\Z}^{m}_{p,q,r}$ for any odd $k$, where $m-1\geq k\geq 1$, and even $n$ by Lemma \ref{lemma_KLV}. 
Let us prove $\{\C^{k}_{p,q,0}\Lambda^{\geq n-(m-1)}_r\}\subseteq\Z^m_{p,q,r}$ for any even $k=0,\ldots,m-2$.
Note that $n-(m-1)\geq 1$; hence, we have $\{\Lambda^{\geq n-(m-1)}_r\C^{m}_{p,q,r}\}\subseteq \C^{\geq n-(m-1)+m}_{p,q,r}=\C^{n+1}_{p,q,r}=\{0\}$ and, similarly, $\{\C^{m}_{p,q,r}\Lambda^{\geq n-(m-1)}_r\}=\{0\}$ by Remark \ref{rem_zero}. Therefore, for any   $k=0,\ldots,n$, 
\begin{eqnarray}
    \{\C^{k}_{p,q,0} \Lambda^{\geq n-(m-1)}_r\} \C^{m}_{p,q,r}= \C^{m}_{p,q,r} \{\C^{k}_{p,q,0} \Lambda^{\geq n-(m-1)}_r\}=\{0\}.
\end{eqnarray}
Thus, 
$\{\C^{k}_{p,q,0}\Lambda^{\geq n-(m-1)}_r\}\subseteq\check{\Z}^m_{p,q,r}$. 
Let us prove that the left set is a subset of the right one in (\ref{need_pr_1}). 
Using $\langle\check{\Z}^m_{p,q,r}\rangle_{(1)}\subseteq\langle\check{\Z}^{m+1}_{p,q,r}\rangle_{(1)}$ for any even $m$ by Lemma \ref{zm_zm1} and applying (\ref{ch_cc_odd}) proved above,  we get
\begin{eqnarray*}
\langle\check{\Z}^m_{p,q,r}\rangle_{(1)}\!\subseteq\!\langle\Lambda^{\leq n-(m+1)}_r\!\oplus\!\!\!\!\!\!\!\bigoplus_{k=0\mod{2}}^{m}\!\!\!\!\!\!\! \{\C^k_{p,q,0}\Lambda_r^{\geq n-(m+1)}\}\!\oplus\!\!\!\!\!\!\!\bigoplus_{k=1\mod{2}}^{m-1}\!\!\!\!\!\!\!\{\C^{k}_{p,q,0}\Lambda_r^{\geq n-m}\}\rangle_{(1)}.
\end{eqnarray*}
Now let us show that the inclusion above implies the inclusion of the left set in the right one in (\ref{need_pr_1}).
The projection of $\langle\check{\Z}^m_{p,q,r}\rangle_{(1)}$ onto the subspace $\langle\Lambda^{\leq n-(m+1)}_r\rangle_{(1)}$ equals zero, since for any odd basis element $X\in\Lambda^{\leq n-(m+1)}_r$ there exists such an even basis element $V\in\C^{m}_{p,q,r}$ that $XV\neq0$ and $XV=VX$. For example, in the case $n=r=4$ and $m=2$, for $X=e_{1}$, we have $V=e_{23}$, and $e_{1}e_{23}=e_{23}e_{1}\neq e_{23}\widehat{e_1}$.
The projection of $\langle\check{\Z}^m_{p,q,r}\rangle_{(1)}$ onto $\langle\{\C^{k}_{p,q,0}\Lambda^{n-m}_r\}\rangle_{(1)}$ equals zero for any even $k$, $m$ and odd $n$ by Lemma \ref{lemma_KLV}. The projection of $\langle\check{\Z}^m_{p,q,r}\rangle_{(1)}$ onto $\langle\{\C^{k}_{p,q,0}\Lambda^{n-m-1}_r\}\rangle_{(1)}$ equals zero for any even $k$, where $m\geq k>0$, since for any basis elements $K=e_{a_1\ldots a_k}\in\C^{k}_{p,q,0}\subset\C^{(0)}_{p,q,r}$ and $L\in\Lambda^{n-m-1}_r\subseteq\C^{(1)}_{p,q,r}$, there exists such an even grade-$m$ element $M\in e_{a_1\ldots a_k}\C^{m-k}_{p,q,r}$ that $LM\neq0$, $LM=ML$, and $KM=MK$, so we get $(KL)M=M(KL)\neq M\widehat{(KL)}$, where we use Remark \ref{rem_zero}. For example, if $n=6$, $p=k=2$, and $r=m=4$, for $K=e_{12}\in\C^{2}_{2,0,0}$ and $L=e_{3}\in\Lambda^{1}_4$, there exists $M=e_{1245}\in e_{12}\C^{2}_{2,0,4}$, such that $(e_{12}e_3)e_{1245}=e_{1245}(e_{12}e_{3})\neq e_{1245}\widehat{(e_{12}e_{3})}$. Finally, the projection of $\check{\Z}^m_{p,q,r}$ onto $\{\C^{m}_{p,q,0}\Lambda^{\geq n-(m-1)}_r\}=\{0\}$ equals zero as well. Thus, we obtain (\ref{ch_cc_m_even}), and the proof is completed.

Finally, let us prove (\ref{cc_odd}). Suppose $m$ is odd and $n\geq m\geq 1$. We have $\langle \Z^m_{p,q,r}\rangle_{(0)}=\langle \check{\Z}^m_{p,q,r}\rangle_{(0)}$ (\ref{same_even}), so we only need to prove 
\begin{eqnarray}
    \!\!\!\langle\Z^m_{p,q,r}\rangle_{(1)}\!=\!\langle \!\!\!\!\!\!\bigoplus_{k=0\mod{2}}^{m-3}\!\!\!\!\!\!\{\C^{k}_{p,q,0}\Lambda^{\geq n-(m-1)}_r\!\}\!\oplus \!\!\!\!\!\!\!\bigoplus_{k=1\mod{2}}^{m-2}\!\!\!\!\!\!\{\C^k_{p,q,0}\Lambda^{\geq n-m}_r\!\}\!\oplus\!\C^{n}_{p,q,r}\rangle_{(1)}.\label{pr_f_1}
\end{eqnarray}
We obtain that the right set is a subset of the left one in (\ref{pr_f_1}), using $\{\C^{m}_{p,q,r}\Lambda^{\geq n-(m-1)}_r\}=\{0\}$, Lemma \ref{lemma_KLV}, and $\langle \C^{n}_{p,q,r}\rangle_{(1)}\subset\Z_{p,q,r}$.
Let us prove that the left set is a subset of the right one in (\ref{pr_f_1}). Using $\langle\Z^m_{p,q,r}\rangle_{(1)}\subseteq\langle\Z^{m+1}_{p,q,r}\rangle_{(1)}$ by Lemma \ref{zl_zm} and applying (\ref{ccm_even}) proved above, we get
\begin{eqnarray*}
    \langle\Z^m_{p,q,r}\rangle_{(1)}&\subseteq&
    \langle \Lambda^{\leq n-m-2}_r\oplus \bigoplus_{k=0\mod{2}}^{m-1}\{\C^{k}_{p,q,0}\Lambda^{\geq n-m-1}_r\}
    \\
    &&\oplus\bigoplus_{k=1\mod{2}}^{m-2}\{\C^{k}_{p,q,0}\Lambda^{\geq n-m}_r \}\oplus \C^{n}_{p,q,r}\rangle_{(1)}.
\end{eqnarray*}
The projection of $\langle\Z^m_{p,q,r}\rangle_{(1)}$ onto the subspace $\langle \Lambda^{\leq n-m-2}_r\rangle_{(1)}$ equals zero because for any odd $X\in\Lambda^{\leq n-m-2}_r$ there exists such an odd $V\in\C^{m}_{p,q,r}$ that $XV\neq0$ and $XV=-VX$. For example, if $n=r=4$ and $m=1$, for $X=e_{1}$, we have $V=e_2$, and $e_{1}e_{2}=-e_{2}e_{1}$. The projection of  $\langle\Z^m_{p,q,r}\rangle_{(1)}$ onto $\{\C^{k}_{p,q,0}\Lambda^{n-m}_r\}$ for any even  $k$ and odd $m$, $n$ is zero by Lemma \ref{lemma_KLV}. The projection of $\langle\Z^m_{p,q,r}\rangle_{(1)}$ onto $\langle\{\C^{k}_{p,q,0}\Lambda^{n-m-1}_r\}\rangle_{(1)}$ for any even $k\leq m-1$ equals zero because for any basis elements $K=e_{a_1\ldots a_k}\in\C^{k}_{p,q,0}\subseteq\C^{(0)}_{p,q,r}$ and $L\in\Lambda^{n-m-1}_r\subseteq\C^{(1)}_{p,q,r}$, there exists such an odd grade-$m$ element $M\in e_{a_1\ldots a_k}\C^{m-k}_{p,q,r}$ that $LM\neq0$, $LM=M\widehat{L}$, and $KM=MK$, therefore, $(KL)M=KM\widehat{L}=M(K\widehat{L})\neq M(KL)$.  For example, if $n=5$, $p=k=2$, and $r=m=3$, for $K=e_{12}\in\C^{2}_{2,0,0}$ and $L=e_{3}\in\Lambda^{1}_3$, we can take $M=e_{124}\in e_{12}\C^{1}_{2,0,3}$ and get $(e_{12}e_3)e_{124}=-e_{124}(e_{12}e_{3})$. Finally, $\{\C^{m-1}_{p,q,0}\Lambda^{\geq n-(m-1)}_r\}=\C^{n}_{p,q,r}$. Thus, we obtain  (\ref{cc_odd}), and the proof is completed.
\end{proof}

\begin{rem}\label{CC3CC4}
    Note that Theorem  \ref{theorem_cc_k_even} implies
\begin{eqnarray}
\!\!\!\!\!\!\!\!\!\!\!\!\!\!&\Z^{m}_{p,q,r}\subseteq\Z^{m+2}_{p,q,r},\quad \check{\Z}^m_{p,q,r}\subseteq\check{\Z}^{m+2}_{p,q,r},\qquad m=1,\ldots, n-2;\label{CC3CC4_2_1}
\\
\!\!\!\!\!\!\!\!\!\!\!\!\!\!& \check{\Z}^m_{p,q,r}\subseteq\Z^{m+1}_{p,q,r},\qquad \Z^m_{p,q,r}\subseteq\Z^{m+1}_{p,q,r},\qquad \mbox{$m$ is odd};\label{CC3CC4_2_2}
\\
\!\!\!\!\!\!\!\!\!\!\!\!\!\!& \check{\Z}^m_{p,q,r}\subseteq \Z^{m+2}_{p,q,r},\qquad\mbox{$m$ is even}.\label{CC3CC4_2_3}
\end{eqnarray}
Using (\ref{CC3CC4_2_1})--(\ref{CC3CC4_2_3}), we get
\begin{eqnarray}
\!\!\!\!\!\!\!\!&\Z^{m}_{p,q,r}\subseteq\Z^{4}_{p,q,r},\qquad \check{\Z}^{m}_{p,q,r}\subseteq\Z^{4}_{p,q,r},\qquad m=1,2,3.\label{CC3CC4_3}
\end{eqnarray}
If $r\leq n-(m+1)$, then
    \begin{eqnarray}
        &&\Z^m_{p,q,r}=\Z^1_{p,q,r},\quad \check{\Z}^m_{p,q,r}=\check{\Z}^1_{p,q,r},\quad \mbox{$m$ is odd};
        \\
        &&\Z^m_{p,q,r}=\Z^2_{p,q,r},\quad\check{\Z}^m_{p,q,r}=\check{\Z}^2_{p,q,r},\quad \mbox{$m$ is even},\;\; m\neq0.
    \end{eqnarray}
We use these relations to prove Theorems \ref{centralizers_qt} and \ref{centralizers_qt_ds}.
\end{rem}

\section{Particular Cases of Centralizers and Twisted Centralizers}\label{section_examples}

In this section, we consider the centralizers and twisted centralizers in the particular cases that are important for applications. In Remarks \ref{remark_CmCm_pq0} and \ref{remark_00n} below, we explicitly write out $\Z^m_{p,q,r}$ and $\check{\Z}^m_{p,q,r}$, $m=0,1,\ldots,n$, in the cases of the non-degenerate Clifford algebra $\C_{p,q,0}$ and the Grassmann algebra $\C_{0,0,n}$ respectively. Note that in these special cases, the centralizers and twisted centralizers have a much simpler form than in the general case of arbitrary $\C_{p,q,r}$ (Theorem \ref{theorem_cc_k_even}).

\begin{rem}\label{remark_CmCm_pq0} In the particular case of the non-degenerate algebra $\C_{p,q,0}$, we get from Theorem \ref{theorem_cc_k_even}
\begin{eqnarray*}
\Z^m_{p,q,0}=
\left\lbrace
\begin{array}{lll}
\C_{p,q,0},&& m=0;\quad m=n\quad \mbox{and}\quad m,n=1\mod{2};
\\
\C^{(0)}_{p,q,0},&& m=n\quad \mbox{and}\quad m,n=0\mod{2};
\\
\C^{0n}_{p,q,0},&&m\neq 0,n\quad \mbox{and}\quad m=0\mod{2} 
\\
&&\mbox{or}\quad m\neq n\quad \mbox{and}\quad m,n=1\mod{2};
\\
\C^{0},&& m=1\mod{2}\quad \mbox{and}\quad n=0\mod{2};
\end{array}
    \right.
\end{eqnarray*}
and
\begin{eqnarray*}
\check{\Z}^m_{p,q,0}\!=\!
\left\lbrace
\begin{array}{lll}
\!\!\C_{p,q,0},&&\!\! m=n,\quad m,n=0\mod{2};
\\
\!\!\C^{(0)}_{p,q,0},&&\!\! m=0;\quad m=n\quad \mbox{and}\quad m,n=1\mod{2};
\\
\!\!\C^{0n}_{p,q,0},&&\!\! m\neq 0,n\quad \mbox{and}\quad m,n=0\mod{2};
\\
\!\!\C^0,&&\!\!m\neq n\quad \mbox{and}\quad m=1\mod{2}
\\
&&\!\!\mbox{or}\quad m\neq 0,\quad  m=0\mod{2}, \quad \mbox{and}\quad n=1\mod{2}.
\end{array}
    \right.
\end{eqnarray*}
\end{rem}

\begin{rem}\label{remark_00n} In the particular case of the Grassmann algebra $\C_{0,0,n}=\Lambda_n$, we obtain from Theorem \ref{theorem_cc_k_even}
\begin{eqnarray*}
&\Z^0_{0,0,n}=\Lambda_n,\qquad \check{\Z}^0_{0,0,n}=\Lambda^{(0)}_{n};
\\
&\Z^m_{0,0,n}=\Lambda_n,\quad \check{\Z}^m_{0,0,n}=\Lambda^{(0)}_n \oplus\langle\Lambda^{\geq n-m+1}_n\rangle_{(1)},\quad m=0\mod{2},\quad m\neq0;
\\
&\Z^m_{0,0,n}=\Lambda^{(0)}_n\oplus \langle\Lambda^{\geq n-m+1}_n\rangle_{(1)},\quad \check{\Z}^m_{0,0,n}=\Lambda_n,\quad m=1\mod{2}.
\end{eqnarray*}

\end{rem}

In Remark \ref{cases_we_use} below, we explicitly write out the particular cases of Theorem~\ref{theorem_cc_k_even} and Remark \ref{remark_00n} in the case of small $k\leq 4$. We use these centralizers and twisted centralizers in Theorems \ref{centralizers_qt} and \ref{centralizers_qt_ds}.
Note that some of these centralizers and twisted centralizers are considered, for instance, in the papers \cite{cNN,br1,GenSpin,OnSomeLie,ICACGA}.
The cases $\Z^{2}_{p,q,r}$ (\ref{cc_2}) and $\check{\Z}^1_{p,q,r}$ (\ref{ch_cc_1}) are proved in detail, for example, in \cite{OnSomeLie}. The other cases are presented for the first time. 

\begin{rem}\label{cases_we_use}
We have:
\begin{eqnarray}
\!\!\!\!\!\!\!\!\!\!\!\!\!\!\!\!\!\!&&\Z^1_{p,q,r}=\Z_{p,q,r}=\left\lbrace
    \begin{array}{lll}
    \Lambda^{(0)}_r\oplus\C^{n}_{p,q,r},&&\mbox{$n$ is odd},
    \\
    \Lambda^{(0)}_r,&&\mbox{$n$ is even};
     \end{array}
    \right.\label{cc_1}
\\
\!\!\!\!\!\!\!\!\!\!\!\!\!\!\!\!\!\!&&\Z^2_{p,q,r}=
\left\lbrace
    \begin{array}{lll}
\Lambda_r\oplus\C^{n}_{p,q,r},&& r\neq n,
\\
\Lambda_r,&& r=n;
\end{array}
    \right.\label{cc_2}
\\
\!\!\!\!\!\!\!\!\!\!\!\!\!\!\!\!\!\!&&\Z^3_{p,q,r}\!=\!\left\lbrace
            \begin{array}{lll}\label{cc_3}
            \Lambda^{(0)}_r\!\oplus\!\Lambda^{n-2}_r\oplus \{\C^{1}_{p,q,0}(\Lambda^{n-3}_r\oplus\Lambda^{n-2}_r)\}
            \\
            \quad\oplus \{\C^{2}_{p,q,0}\Lambda^{n-3}_r\}\oplus\C^{n}_{p,q,r}, &&\mbox{$n$ is odd},
            \\
            \\
            \Lambda^{(0)}_r\!\oplus\!\Lambda^{n-1}_r\oplus \{\C^{1}_{p,q,0}\Lambda^{\geq n-2}_r\}\oplus \{\C^{2}_{p,q,0}\Lambda^{n-2}_r\}, &&\mbox{$n$ is even};
            \end{array}
            \right.
\\
\!\!\!\!\!\!\!\!\!\!\!\!\!\!\!\!\!\!&&\Z^4_{p,q,r}\!=\!
\left\lbrace
\begin{array}{lll}\label{cc_4}
\Lambda_r\oplus\{\C^{1}_{p,q,0}(\Lambda^{n-3}_r\oplus\Lambda^{n-2}_r)\}
\\
\quad\oplus \{\C^{2}_{p,q,0}(\Lambda^{n-4}_r\oplus\Lambda^{n-3}_r)\}\oplus\C^{n}_{p,q,r},&& r\neq n,
\\
\\
\Lambda_r, && r=n;
\end{array}
\right.
\end{eqnarray}
and:
\begin{eqnarray}
    \!\!\!\!\!\!\!\!\!\!\!\!\!\!\!\!\!\!&&\check{\Z}^1_{p,q,r}=\Lambda_r;\label{ch_cc_1}
    \\
    \!\!\!\!\!\!\!\!\!\!\!\!\!\!\!\!\!\!&&\check{\Z}^2_{p,q,r}=
    \left\lbrace
    \begin{array}{lll}\label{ch_cc_2}
    \Lambda^{(0)}_r\oplus\Lambda^{n}_r\oplus \{\C^{1}_{p,q,0}\Lambda^{n-1}_r\},&&\!\!\!\!\!\!\mbox{$n$ is odd},
    \\
    \Lambda^{(0)}_r\oplus\Lambda^{n-1}_r\oplus \{\C^{1}_{p,q,0}\Lambda^{n-2}_r\}\oplus\C^{n}_{p,q,r},&&\!\!\!\!\!\!\mbox{$n$ is even},\;\; r\neq n;
    \\
    \Lambda^{(0)}_r\oplus\Lambda^{n-1}_r,&&\!\!\!\!\!\!\mbox{$n$ is even},\;\; r=n;
    \end{array}
    \right.
    \\
    \!\!\!\!\!\!\!\!\!\!\!\!\!\!\!\!\!\!&&\check{\Z}^3_{p,q,r}=\Lambda_r\oplus \{\C^{1}_{p,q,0}\Lambda^{\geq n-2}_r\}\oplus \{\C^{2}_{p,q,0}\Lambda^{\geq n-3}_r\};\label{ch_cc_3}
    \\
    \!\!\!\!\!\!\!\!\!\!\!\!\!\!\!\!\!\!&&\check{\Z}^4_{p,q,r}=
    \left\lbrace
    \begin{array}{lll}
    \Lambda^{(0)}_r \oplus \Lambda^{n-2}_r\oplus\Lambda^n_r\oplus \{\C^1_{p,q,0}\Lambda^{\geq n-3}_r\}
    \\
    \quad \oplus \{\C^2_{p,q,0}\Lambda^{\geq n-3}_r\}\oplus\{\C^3_{p,q,0}\Lambda^{n-3}_r\},\!\!\!\!\!\!\!\!&& \mbox{$n$ is odd},
    \\
    \\
    \Lambda^{(0)}_r \oplus\Lambda^{n-3}_r\oplus\Lambda^{n-1}_r\oplus\{\C^3_{p,q,0}\Lambda^{n-4}_r\}
    \\
    \quad \oplus \{\C^{2}_{p,q,0}(\Lambda^{n-4}_r\oplus\Lambda^{n-3}_r)\}\oplus\C^{n}_{p,q,r}
    \\
    \quad\oplus
    \{\C^1_{p,q,0}(\Lambda^{n-4}_r \oplus \Lambda^{n-3}_r \oplus \Lambda^{n-2}_r)\},\!\!\!\!\!\!\!\!&& \mbox{$n$ is even},\;\; r\neq n,
    \\
    \\
    \Lambda^{(0)}_r \oplus\Lambda^{n-3}_r\oplus\Lambda^{n-1}_r,\!\!\!\!\!\!\!\!&& \mbox{$n$ is even},\;\; r=n.
    \end{array}
    \right.
\end{eqnarray}
\end{rem}

In Remark \ref{rem_ad}, we consider how the centralizers and the twisted centralizers are related to the kernels of the adjoint and twisted adjoint representations.
\begin{rem}\label{rem_ad}
    Note that 
\begin{eqnarray}
    \Z^{1\times}_{p,q,r}=\ker(\ad),\qquad \check{\Z}^{1\times}_{p,q,r}=\ker(\tilde{\ad})
\end{eqnarray}
and 
\begin{eqnarray}
    \ker(\ad)\subseteq\Z^{m\times}_{p,q,r},\qquad \ker(\check{\ad})=\Lambda^{(0)\times}_r\subseteq\check{\Z}^{m\times}_{p,q,r},\qquad  m=0,\ldots,n,\label{CC3CC4_2}
\end{eqnarray}
where $\ker(\ad)$, $\ker(\check{\ad})$, and $\ker(\tilde{\ad})$ are the kernels of the adjoint representation $\ad$ and the twisted adjoint  representations $\check{\ad}$ and $\tilde{\ad}$ respectively. 
The adjoint representation $\ad:\C^{\times}_{p,q,r}\rightarrow\Aut(\C_{p,q,r})$ acts on the group of all invertible elements as $T\mapsto\ad_T$, where 
\begin{eqnarray}\label{ar}
\ad_{T}(U)=TU T^{-1},\qquad U\in\C_{p,q,r},\qquad T\in\C^{\times}_{p,q,r}.
\end{eqnarray}
The twisted adjoint representation $\check{\ad}$ has been introduced in a particular case by Atiyah, Bott, and Shapiro in \cite{ABS}. The representation $\check{\ad}:\C^{\times}_{p,q,r}\rightarrow\Aut(\C_{p,q,r})$ acts on $\C^{\times}_{p,q,r}$ as $T\mapsto\check{\ad}_T$ with
\begin{eqnarray}
\check{\ad}_{T}(U)=\widehat{T}U T^{-1},\qquad U\in\C_{p,q,r},\qquad T\in\C^{\times}_{p,q,r}.\label{twa1}
\end{eqnarray}
The representation $\tilde{\ad}:\C^{\times}_{p,q,r}\rightarrow\Aut(\C_{p,q,r})$ acts on $\C^{\times}_{p,q,r}$  as $T\mapsto\tilde{\ad}_T$ with 
\begin{eqnarray}
\tilde{\ad}_{T}(U)=T\langle U\rangle_{(0)} T^{-1}+\widehat{T} \langle U\rangle_{(1)} T^{-1},\quad \forall U\in\C_{p,q,r},\quad T\in\C^{\times}_{p,q,r}.\label{twa22}
\end{eqnarray}
See the details about $\ad$, $\check{\ad}$, $\tilde{\ad}$, and their kernels, for example, in \cite{OnSomeLie}.
\end{rem}

\section{Centralizers and Twisted Centralizers of the Subspaces Determined by the Grade Involution and the Reversion}\label{section_centralizers_qt}

This section finds explicit forms of the centralizers and twisted centralizers of the subspaces $\C^{\overline{m}}_{p,q,r}$ (\ref{qtdef}), $m=0,1,2,3$, determined by the grade involution and the reversion  and their direct sums. In particular, we consider the centralizers and twisted centralizers of the even $\C^{(0)}_{p,q,r}=\C^{\overline{02}}_{p,q,r}$ and odd $\C^{(1)}_{p,q,r}=\C^{\overline{13}}_{p,q,r}$ subspaces.

Let us consider the centralizers $\Z^{\overline{m}}_{p,q,r}$ and twisted centralizers $\check{\Z}^{\overline{m}}_{p,q,r}$ of the subspaces $\C^{\overline{m}}_{p,q,r}$ (\ref{qtdef}), $m=0,1,2,3$, in $\C_{p,q,r}$:
\begin{eqnarray}
\!\!\!\!\!\!\!\!\!\!\!\!\Z^{\overline{m}}_{p,q,r}\!\!\!&:=&\!\!\!\{X\in\C_{p,q,r}:\quad X V = V X,\quad \forall V\in\C^{\overline{m}}_{p,q,r}\},\quad m=0,1,2,3,\label{def_CC_ov}
\\
\!\!\!\!\!\!\!\!\!\!\!\!\check{\Z}^{\overline{m}}_{p,q,r}\!\!\!&:=&\!\!\!\{X\in\C_{p,q,r}:\quad \widehat{X} V = V X,\quad \forall V\in\C^{\overline{m}}_{p,q,r}\},\quad m=0,1,2,3.\label{def_chCC_ov}
\end{eqnarray}
In Theorem \ref{centralizers_qt}, we prove that $\Z^{\overline{m}}_{p,q,r}$ and $\check{\Z}^{\overline{m}}_{p,q,r}$ coincide with some of the centralizers $\Z^m_{p,q,r}$ and the twisted centralizers $\check{\Z}^m_{p,q,r}$  of the subspaces of fixed grades, which are considered in Sections \ref{section_centralizers} and \ref{section_examples}.

\begin{thm}\label{centralizers_qt}
We have
\begin{eqnarray}
&\Z^{\overline{m}}_{p,q,r}=\Z^{m}_{p,q,r},\quad\check{\Z}^{\overline{m}}_{p,q,r}=\check{\Z}^{m}_{p,q,r},\quad m=1,2,3;
\\
&\Z^{\overline{0}}_{p,q,r}=\Z^{4}_{p,q,r},\quad \check{\Z}^{\overline{0}}_{p,q,r}=\langle\Z^{4}_{p,q,r}\rangle_{(0)}.\label{f_f3}
\end{eqnarray}
\end{thm}
The centralizers $\Z^m_{p,q,r}$ and twisted centralizers  $\check{\Z}^m_{p,q,r}$, $m=1,2,3,4$, are written out explicitly in Remark \ref{cases_we_use} for the readers' convenience. In the formula (\ref{f_f3}), we have
\begin{eqnarray*}
\langle\Z^{4}_{p,q,r}\rangle_{(0)} \!=\! 
\left\lbrace
    \begin{array}{lll}
\!\!\!\!\Lambda^{(0)}_r\!\oplus \!\{\C^{1}_{p,q,0}\Lambda^{n-2}_r\}\!\oplus 
\!\{\C^{2}_{p,q,0}\Lambda^{n-3}_r\},\; \mbox{$n$ is odd or $r=n$,}
\\
\!\!\!\!\Lambda^{(0)}_r\!\oplus \!\{\C^{1}_{p,q,0}\Lambda^{n-3}_r\}\!\oplus \!\{\C^{2}_{p,q,0}\Lambda^{n-4}_r\}\!\oplus\!\C^{n}_{p,q,r},\;\mbox{$n$ is even, $r\neq n$}.
\end{array}
    \right.
\end{eqnarray*}

\begin{proof}
The inclusions $\Z^{\overline{m}}_{p,q,r}\subseteq\Z^{m}_{p,q,r}$,  $\check{\Z}^{\overline{m}}_{p,q,r}\subseteq\check{\Z}^m_{p,q,r}$, $m=1,2,3$,  and $\Z^{\overline{0}}_{p,q,r}\subseteq\Z^{4}_{p,q,r}$ follow from $\C^{k}_{p,q,r}\subseteq\C^{\overline{k}}_{p,q,r}$, $k=1,2,3$, and $\C^{4}_{p,q,r}\subset\C^{\overline{0}}_{p,q,r}$ respectively. We get $\check{\Z}^{\overline{0}}_{p,q,r}\subseteq\check{\Z}^4_{p,q,r}\cap\check{\Z}^0_{p,q,r}=\Z^4_{p,q,r}\cap\C^{(0)}_{p,q,r}$, using $\C^{04}_{p,q,r}\subseteq\C^{\overline{0}}_{p,q,r}$ and $\check{\Z}^0_{p,q,r}=\C^{(0)}_{p,q,r}$ (\ref{chCC_0}).

Let us prove $\Z^{m}_{p,q,r}\subseteq\Z^{\overline{m}}_{p,q,r}$ and $\check{\Z}^m_{p,q,r}\subseteq\check{\Z}^{\overline{m}}_{p,q,r}$, $m=1,2,3$. Any basis element of $\C^{\overline{k}}_{p,q,r}$, $k=1,2,3$,  can be represented as a product of one basis element of $\C^{k}_{p,q,r}$ and basis elements of $\C^{4}_{p,q,r}$. Since $\Z^{m}_{p,q,r}\subseteq \Z^4_{p,q,r}$ and $\check{\Z}^m_{p,q,r}\subseteq\Z^4_{p,q,r}$ by the statement (\ref{CC3CC4_3}) of Remark \ref{CC3CC4}, we get $\Z^{m}_{p,q,r}\subseteq\Z^{\overline{m}}_{p,q,r}$ and $\check{\Z}^m_{p,q,r}\subseteq\check{\Z}^{\overline{m}}_{p,q,r}$.

We obtain $\Z^4_{p,q,r}\subseteq\Z^{\overline{0}}_{p,q,r}$ and $\Z^4_{p,q,r}\cap\C^{(0)}_{p,q,r}\subseteq\check{\Z}^{\overline{0}}_{p,q,r}$ because $\Z^4_{p,q,r}\subseteq\Z^0_{p,q,r}$ and $\Z^4_{p,q,r}\cap\C^{(0)}_{p,q,r}\subseteq\C^{(0)}_{p,q,r}=\check{\Z}^0_{p,q,r}$ (\ref{chCC_0}) respectively and any basis element of $\C^{\overline{0}}_{p,q,r}\setminus\C^0$ can be represented as a product of basis elements of $\C^{4}_{p,q,r}$. 
\end{proof}

Let us denote by $\Z^{\overline{km}}_{p,q,r}$ and $\check{\Z}^{\overline{km}}_{p,q,r}$, $k,m=0,1,2,3$, the centralizers and the twisted centralizers respectively of the direct sums of the subspaces $\C^{\overline{m}}_{p,q,r}$~(\ref{qtdef}) in $\C_{p,q,r}$:
\begin{eqnarray*}
\Z^{\overline{km}}_{p,q,r}&:=&\Z^{\overline{k}}_{p,q,r}\cap\Z^{\overline{m}}_{p,q,r}=\{X\in\C_{p,q,r}:\quad X V = V X,\quad \forall V\in\C^{\overline{km}}_{p,q,r}\},
\\
\check{\Z}^{\overline{km}}_{p,q,r}&:=&\check{\Z}^{\overline{k}}_{p,q,r}\cap\check{\Z}^{\overline{m}}_{p,q,r}=\{X\in\C_{p,q,r}:\quad \widehat{X} V = V X,\quad \forall V\in\C^{\overline{km}}_{p,q,r}\}.
\end{eqnarray*}
Note that $\Z^{\overline{02}}_{p,q,r}$, $\check{\Z}^{\overline{02}}_{p,q,r}$, $\Z^{\overline{13}}_{p,q,r}$,  and $\check{\Z}^{\overline{13}}_{p,q,r}$ are the centralizers and the twisted centralizers of the even $\C^{(0)}_{p,q,r}$ and odd $\C^{(1)}_{p,q,r}$ subspaces respectively.
In Theorem \ref{centralizers_qt_ds}, we find explicit forms of $\Z^{\overline{km}}_{p,q,r}$ and $\check{\Z}^{\overline{km}}_{p,q,r}$, $k,m=0,1,2,3$.

\begin{thm}\label{centralizers_qt_ds}
We have
\begin{eqnarray}
    &\!\!\!\!\!\!\Z^{\overline{01}}_{p,q,r}=\Z^{\overline{12}}_{p,q,r}=\Z^{\overline{13}}_{p,q,r}=\Z_{p,q,r},\quad \Z^{\overline{23}}_{p,q,r}=\Z^2_{p,q,r}\cap\Z^3_{p,q,r},\label{centralizers_qt_ds_1}
    \\
    &\Z^{\overline{02}}_{p,q,r}=\Z^2_{p,q,r},\quad \Z^{\overline{03}}_{p,q,r}=\Z^3_{p,q,r},\label{centralizers_qt_ds_2_00}
    \\
    &\!\!\!\!\!\! \check{\Z}^{\overline{12}}_{p,q,r}=\check{\Z}^1_{p,q,r}\cap\check{\Z}^2_{p,q,r},\quad \check{\Z}^{\overline{23}}_{p,q,r}=\check{\Z}^2_{p,q,r}\cap\check{\Z}^3_{p,q,r},\quad \check{\Z}^{\overline{13}}_{p,q,r}=\check{\Z}^1_{p,q,r}, \label{centralizers_qt_ds_2_0}
    \\
    &\check{\Z}^{\overline{01}}_{p,q,r}=\langle{\Z}^1_{p,q,r}\rangle_{(0)},\quad \check{\Z}^{\overline{02}}_{p,q,r}=\langle{\Z}^2_{p,q,r}\rangle_{(0)},\quad\check{\Z}^{\overline{03}}_{p,q,r}=\langle\Z^3_{p,q,r}\rangle_{(0)}.\label{centralizers_qt_ds_2}
\end{eqnarray}
\end{thm}
The centralizers $\Z^2_{p,q,r}$, $\Z^3_{p,q,r}$, $\Z_{p,q,r}=\Z^1_{p,q,r}$ and the twisted centralizer $\check{\Z}^1_{p,q,r}$ are written out explicitly in Remark \ref{cases_we_use}.
In the formulas (\ref{centralizers_qt_ds_1})--(\ref{centralizers_qt_ds_2}), we have
\begin{eqnarray*}
     \!\!\!\!\!\!&\Z^2_{p,q,r}\cap\Z^3_{p,q,r}=
     \left\lbrace
    \begin{array}{lll}
    \!\!\!\Lambda^{(0)}_r\oplus \Lambda^{n-2}_r\oplus\C^{n}_{p,q,r}, \!\!\!\!\!\!\!\!\!&&\mbox{$n$ is odd},
    \\
    \!\!\!\Lambda^{(0)}_r\oplus\Lambda^{n-1}_r\oplus\{\C^{1}_{p,q,0}\Lambda^{n-1}_r\}\oplus\{\C^{2}_{p,q,0}\Lambda^{n-2}_r\},\!\!\!\!\!\!\!\!\!&&\mbox{$n$ is even};
    \end{array}
    \right.
     \\
     \!\!\!\!\!\!&\check{\Z}^1_{p,q,r}\cap\check{\Z}^2_{p,q,r} =
      \left\lbrace
    \begin{array}{lll}
    \Lambda^{(0)}_r\oplus\Lambda^n_r,&&\mbox{$n$ is odd},
    \\
    \Lambda^{(0)}_r\oplus\Lambda^{n-1}_r,&&\mbox{$n$ is even};
    \end{array}
    \right.
     \end{eqnarray*}
    \begin{eqnarray*}
     \!\!\!\!\!\!&\check{\Z}^2_{p,q,r}\!\cap\check{\Z}^3_{p,q,r} \!=\!
     \left\lbrace
    \begin{array}{lll}
    \!\!\!\Lambda^{(0)}_r\!\oplus\!\Lambda^{n}_r\!\oplus\! \{\C^{1}_{p,q,0}\Lambda^{n-1}_r\},\!\!\!\!\!\!\!\!\!&&\mbox{$n$ is odd},
    \\
    \!\!\!\Lambda^{(0)}_r\!\oplus\!\Lambda^{n-1}_r\!\oplus\! \{\C^{1}_{p,q,0}\Lambda^{\geq n-2}_r\}\!\oplus\!\{\C^{2}_{p,q,0}\Lambda^{n-2}_r\},\!\!\!\!\!\!\!\!\!&&\mbox{$n$ is even};
    \end{array}
    \right.
     \\
    \!\!\! \!\!\!&\langle \Z^1_{p,q,r}\rangle_{(0)}=\Lambda^{(0)}_r,\quad 
     \langle \Z^2_{p,q,r}\rangle_{(0)} =
     \left\lbrace
    \begin{array}{lll}
    \!\!\!\Lambda^{(0)}_r,\!\!\!\!\!\!\!\!&&\mbox{$n$ is odd};\;\; \mbox{$n$ is even, $r=n$},
    \\
    \!\!\!\Lambda^{(0)}_r\oplus\C^{n}_{p,q,r},\!\!\!\!\!\!\!\!&&\mbox{$n$ is even, $r\neq n$};
    \end{array}
    \right.
    \\
    \!\!\!\!\!\!&\langle \Z^3_{p,q,r}\rangle_{(0)} = 
    \left\lbrace
            \begin{array}{lll}
            \Lambda^{(0)}_r\oplus \{\C^{1}_{p,q,0}\Lambda^{n-2}_r\}\oplus \{\C^{2}_{p,q,0}\Lambda^{n-3}_r\}, &&\mbox{$n$ is odd},
            \\
            \Lambda^{(0)}_r\!\oplus \{\C^{1}_{p,q,0}\Lambda^{n-1}_r\}\oplus \{\C^{2}_{p,q,0}\Lambda^{n-2}_r\}, &&\mbox{$n$ is even}.
            \end{array}
            \right.
\end{eqnarray*}

\begin{proof} First, let us prove (\ref{centralizers_qt_ds_1}) and (\ref{centralizers_qt_ds_2_00}).
    We get 
 \begin{eqnarray}
 \Z^{\overline{23}}_{p,q,r}=\Z^{\overline{2}}_{p,q,r}\cap\Z^{\overline{3}}_{p,q,r}=\Z^2_{p,q,r}\cap\Z^3_{p,q,r}
\end{eqnarray}
by Theorem \ref{centralizers_qt}.
 For $k=1,2,3$, we obtain
 \begin{eqnarray}
     \Z^{\overline{0k}}_{p,q,r}=\Z^{\overline{0}}_{p,q,r}\cap \Z^{\overline{k}}_{p,q,r}=\Z^4_{p,q,r}\cap\Z^k_{p,q,r}=\Z^k_{p,q,r},
 \end{eqnarray}
 using  $\Z^{\overline{k}}_{p,q,r}=\Z^k_{p,q,r}$,  $\Z^{\overline{0}}_{p,q,r}=\Z^4_{p,q,r}$ by Theorem \ref{centralizers_qt} and $\Z^k_{p,q,r}\subseteq\Z^4_{p,q,r}$ by Remark \ref{CC3CC4}. Since $\Z^1_{p,q,r}=\Z_{p,q,r}$ by Remark \ref{cases_we_use}, we get $\Z^{\overline{01}}_{p,q,r}=\Z_{p,q,r}$. Similarly, for $l=2,3$, we obtain
 \begin{eqnarray}
     \Z^{\overline{1l}}_{p,q,r} = \Z^{\overline{1}}_{p,q,r}\cap\Z^{\overline{l}}_{p,q,r}=\Z^1_{p,q,r}\cap\Z^{l}_{p,q,r}=\Z_{p,q,r}\cap\Z^{l}_{p,q,r}=\Z_{p,q,r},
 \end{eqnarray}
 where we use $\Z_{p,q,r}\subseteq\Z^{2}_{p,q,r}$ and $\Z_{p,q,r}\subseteq\Z^3_{p,q,r}$ by the formula (\ref{CC3CC4_2_1}) of Remark \ref{CC3CC4}.

Let us prove (\ref{centralizers_qt_ds_2_0}). We get 
\begin{eqnarray*}
\check{\Z}^{\overline{12}}_{p,q,r}\!=\!\check{\Z}^{\overline{1}}_{p,q,r}\!\cap\!\check{\Z}^{\overline{2}}_{p,q,r}\!=\!\check{\Z}^1_{p,q,r}\!\cap\!\check{\Z}^2_{p,q,r}, \;\;\check{\Z}^{\overline{23}}_{p,q,r}\!=\!\check{\Z}^{\overline{2}}_{p,q,r}\!\cap\!\check{\Z}^{\overline{3}}_{p,q,r}\!=\!\check{\Z}^2_{p,q,r}\!\cap\!\check{\Z}^3_{p,q,r},
\end{eqnarray*}
 and $\check{\Z}^{\overline{13}}_{p,q,r}=\check{\Z}^{\overline{1}}_{p,q,r}\cap\check{\Z}^{\overline{3}}_{p,q,r}=\check{\Z}^1_{p,q,r}\cap\check{\Z}^3_{p,q,r}=\check{\Z}^1_{p,q,r}$, using Theorem  \ref{centralizers_qt} and $\check{\Z}^1_{p,q,r}\subseteq\check{\Z}^3_{p,q,r}$ by Remark \ref{CC3CC4}.
    Now we prove (\ref{centralizers_qt_ds_2}).  For $k=1,2,3$, we get
    \begin{eqnarray*}
    \check{\Z}^{\overline{0k}}_{p,q,r}&=&\check{\Z}^{\overline{0}}_{p,q,r}\cap\check{\Z}^{\overline{k}}_{p,q,r}=\Z^{4}_{p,q,r}\cap\C^{(0)}_{p,q,r}\cap\check{\Z}^k_{p,q,r}
    \\
    &=&\check{\Z}^k_{p,q,r}\cap\C^{(0)}_{p,q,r}=\Z^k_{p,q,r}\cap\C^{(0)}_{p,q,r}=\langle{\Z}^k_{p,q,r}\rangle_{(0)},
\end{eqnarray*} 
using Theorem  \ref{centralizers_qt}, $\check{\Z}^k_{p,q,r}\subseteq\Z^4_{p,q,r}$ by Remark \ref{CC3CC4}, and (\ref{same_even}).
\end{proof}

In Table \ref{table_c}, we present a comprehensive list of the centralizers and twisted centralizers of the subspaces determined by the grade involution and the reversion $\C^{\overline{k}}_{p,q,r}$ and their direct sums $\C^{\overline{kl}}_{p,q,r}$, $k,l=0,1,2,3$. The first column indicates different centralizers $\Z^{\overline{k}}_{p,q,r}$ and twisted centralizers $\Z^{\overline{kl}}_{p,q,r}$, $k,l=0,1,2,3$. The second column contains the corresponding (see Theorems \ref{centralizers_qt} and \ref{centralizers_qt_ds}) centralizers $\Z^m_{p,q,r}$ and twisted centralizers $\check{\Z}^m_{p,q,r}$ of the subspaces of fixed grades $\C^m_{p,q,r}$, while the third column contains their explicit forms.
\begin{landscape}
\begin{table}
\caption{Centralizers (C) and twisted centralizers (TC) in $\C_{p,q,r}$}\label{table_c} 
{\footnotesize
\begin{tabular}{| c | c  | c |}
\hline 
C and TC of $\C^{\overline{k}}_{p,q,r}$ and $\C^{\overline{kl}}_{p,q,r}$ & C and TC of $\C^{{m}}_{p,q,r}$ & Explicit forms \\ \hline
$\Z^{\overline{1}}_{p,q,r}=\Z^{\overline{01}}_{p,q,r}=\Z^{\overline{12}}_{p,q,r}=\Z^{\overline{13}}_{p,q,r}$ & $\Z_{p,q,r}$ & $\left\lbrace
    \begin{array}{lll}
    \Lambda^{(0)}_r\oplus\C^{n}_{p,q,r},&&\mbox{$n$ is odd},
    \\
    \Lambda^{(0)}_r,&&\mbox{$n$ is even};
     \end{array}
    \right.$ \\  \hline
    $\check{\Z}^{\overline{1}}_{p,q,r}=\check{\Z}^{\overline{13}}_{p,q,r}$ & $\check{\Z}^1_{p,q,r}$ & $\Lambda_r$ \\ \hline
    $\Z^{\overline{2}}_{p,q,r}=\Z^{\overline{02}}_{p,q,r}$ & $\Z^2_{p,q,r}$ & $\left\lbrace
    \begin{array}{lll}
\Lambda_r\oplus\C^{n}_{p,q,r},&& r\neq n,
\\
\Lambda_r,&& r=n;
\end{array}
    \right.$ \\ \hline
$\check{\Z}^{\overline{2}}_{p,q,r}$ & $\check{\Z}^{2}_{p,q,r}$ & $\left\lbrace
    \begin{array}{lll}
    \Lambda^{(0)}_r\oplus\Lambda^{n}_r\oplus \{\C^{1}_{p,q,0}\Lambda^{n-1}_r\},&&\mbox{$n$ is odd},
    \\
    \Lambda^{(0)}_r\oplus\Lambda^{n-1}_r\oplus \{\C^{1}_{p,q,0}\Lambda^{n-2}_r\}\oplus\C^{n}_{p,q,r},&&\mbox{$n$ is even},\;\; r\neq n,
    \\
    \Lambda^{(0)}_r\oplus\Lambda^{n-1}_r,&&\mbox{$n$ is even},\;\; r=n;
    \end{array}
    \right.$\\ \hline
    $\Z^{\overline{3}}_{p,q,r}=\Z^{\overline{03}}_{p,q,r}$ & $\Z^3_{p,q,r}$ & $\left\lbrace
            \begin{array}{lll}
            \Lambda^{(0)}_r\!\oplus\!\Lambda^{n-2}_r\oplus \{\C^{1}_{p,q,0}(\Lambda^{n-3}_r\oplus\Lambda^{n-2}_r)\}
            \\
            \quad\oplus \{\C^{2}_{p,q,0}\Lambda^{n-3}_r\}\oplus\C^{n}_{p,q,r}, &&\mbox{$n$ is odd},
            \\
            \Lambda^{(0)}_r\!\oplus\!\Lambda^{n-1}_r\oplus \{\C^{1}_{p,q,0}\Lambda^{\geq n-2}_r\}\oplus \{\C^{2}_{p,q,0}\Lambda^{n-2}_r\}, &&\mbox{$n$ is even};
            \end{array}
            \right.$ \\ \hline
$\check{\Z}^{\overline{3}}_{p,q,r}$ & $\check{\Z}^{3}_{p,q,r}$ & $\Lambda_r\oplus \{\C^{1}_{p,q,0}\Lambda^{\geq n-2}_r\}\oplus \{\C^{2}_{p,q,0}\Lambda^{\geq n-3}_r\}$\\ \hline
$\Z^{\overline{0}}_{p,q,r}$ & $\Z^{4}_{p,q,r}$ & $\left\lbrace
\begin{array}{lll}
\!\!\!\Lambda_r\!\oplus\! \{\C^{1}_{p,q,0}(\Lambda^{n-3}_r\!\oplus\!\Lambda^{n-2}_r)\}\!\oplus\! \{\C^{2}_{p,q,0}(\Lambda^{n-4}_r\!\oplus\!\Lambda^{n-3}_r)\}\!\oplus\!\C^{n}_{p,q,r}, \!\!\!\!\!\!\!\!\!&& r\neq n,
\\
\!\!\!\Lambda_r, \!\!\!\!\!\!\!\!\!&& r=n;
\end{array}
\right.$ \\ \hline
$\check{\Z}^{\overline{0}}_{p,q,r}$ & $\langle\Z^{4}_{p,q,r}\rangle_{(0)}$ & $\left\lbrace
    \begin{array}{lll}
\!\!\!\Lambda^{(0)}_r\oplus \{\C^{1}_{p,q,0}\Lambda^{n-2}_r\}\oplus \{\C^{2}_{p,q,0}\Lambda^{n-3}_r\},&&\!\!\!\!\!\!\!\!\! \mbox{$n$ is odd or $r=n$,}
\\
\!\!\!\Lambda^{(0)}_r\oplus \{\C^{1}_{p,q,0}\Lambda^{n-3}_r\}\oplus \{\C^{2}_{p,q,0}\Lambda^{n-4}_r\}\oplus\C^{n}_{p,q,r},&& \!\!\!\!\!\!\!\!\!\mbox{$n$ is even, $r\neq n$};
\end{array}
    \right.$ \\ \hline
$\Z^{\overline{23}}_{p,q,r}$ & $\Z^2_{p,q,r}\cap\Z^3_{p,q,r}$ & $\left\lbrace
    \begin{array}{lll}
    \Lambda^{(0)}_r\oplus \Lambda^{n-2}_r\oplus\C^{n}_{p,q,r}, &&\mbox{$n$ is odd},
    \\
    \Lambda^{(0)}_r\oplus\Lambda^{n-1}_r\oplus\{\C^{1}_{p,q,0}\Lambda^{n-1}_r\}\oplus\{\C^{2}_{p,q,0}\Lambda^{n-2}_r\},&&\mbox{$n$ is even};
    \end{array}
    \right.$\\ \hline
$\check{\Z}^{\overline{12}}_{p,q,r}$ & $\check{\Z}^1_{p,q,r}\cap\check{\Z}^2_{p,q,r}$ & $ \left\lbrace
    \begin{array}{lll}
    \Lambda^{(0)}_r\oplus\Lambda^n_r,&&\mbox{$n$ is odd},
    \\
    \Lambda^{(0)}_r\oplus\Lambda^{n-1}_r,&&\mbox{$n$ is even};
    \end{array}
    \right.$ \\ \hline
$\check{\Z}^{\overline{23}}_{p,q,r}$ & $\check{\Z}^2_{p,q,r}\cap\check{\Z}^3_{p,q,r}$ & $\left\lbrace
    \begin{array}{lll}
    \!\!\!\Lambda^{(0)}_r\!\oplus\!\Lambda^{n}_r\!\oplus\! \{\C^{1}_{p,q,0}\Lambda^{n-1}_r\},\!\!\!\!\!\!\!\!\!&&\mbox{$n$ is odd},
    \\
    \!\!\!\Lambda^{(0)}_r\!\oplus\!\Lambda^{n-1}_r\!\oplus\! \{\C^{1}_{p,q,0}\Lambda^{\geq n-2}_r\}\!\oplus\!\{\C^{2}_{p,q,0}\Lambda^{n-2}_r\},\!\!\!\!\!\!\!\!\!&&\mbox{$n$ is even};
    \end{array}
    \right.$ \\ \hline
$\check{\Z}^{\overline{01}}_{p,q,r}$ & $\langle{\Z}^1_{p,q,r}\rangle_{(0)}$& $\Lambda^{(0)}_r$\\ \hline
$\check{\Z}^{\overline{02}}_{p,q,r}$& $\langle{\Z}^2_{p,q,r}\rangle_{(0)}$ & $\left\lbrace
    \begin{array}{lll}
    \!\!\!\Lambda^{(0)}_r,\!\!\!\!\!\!\!\!&&\mbox{$n$ is odd},\;\; \mbox{$n$ is even, $r=n$},
    \\
    \!\!\!\Lambda^{(0)}_r\oplus\C^{n}_{p,q,r},\!\!\!\!\!\!\!\!&&\mbox{$n$ is even, $r\neq n$};
    \end{array}
    \right.$\\ \hline
$\check{\Z}^{\overline{03}}_{p,q,r}$& $\langle\Z^3_{p,q,r}\rangle_{(0)}$& $\left\lbrace
            \begin{array}{lll}
            \Lambda^{(0)}_r\oplus \{\C^{1}_{p,q,0}\Lambda^{n-2}_r\}\oplus \{\C^{2}_{p,q,0}\Lambda^{n-3}_r\}, &&\mbox{$n$ is odd},
            \\
            \Lambda^{(0)}_r\!\oplus \{\C^{1}_{p,q,0}\Lambda^{n-1}_r\}\oplus \{\C^{2}_{p,q,0}\Lambda^{n-2}_r\}, &&\mbox{$n$ is even}.
            \end{array}
            \right.$\\ \hline
\end{tabular}
}
\end{table}
\end{landscape}

\begin{rem}\label{rem_ds}
The equalities for the centralizer $\Z^{\overline{02}}_{p,q,r}$ and the twisted centralizer $\check{\Z}^{\overline{02}}_{p,q,r}$ of the even subspace presented in the formulas (\ref{centralizers_qt_ds_2_00}) and (\ref{centralizers_qt_ds_2}) are proved in Lemma 3.2 \cite{OnSomeLie} in the case of the degenerate and non-degenerate algebras $\C_{p,q,r}$. In the non-degenerate case  $\C_{p,q,0}$, the set $\Z^{\overline{02}}_{p,q,0}$ is considered, for example, in \cite{OnInner,HelmBook,garling} and the set $\check{\Z}^{\overline{02}}_{p,q,0}$ is considered in \cite{GenSpin}. The other equalities in the formulas 
(\ref{centralizers_qt_ds_1})--(\ref{centralizers_qt_ds_2}) are presented for the first time.
\end{rem}

In Appendix \ref{sec_tild}, we consider the twisted centralizers $\tilde{\Z}^{\overline{m}}_{p,q,r}$ and $\tilde{\Z}^{\overline{km}}_{p,q,r}$ defined as
\begin{eqnarray*} 
    \tilde{\Z}^{\overline{m}}_{p,q,r}&:=&\{X\in\C_{p,q,r}:\quad X\langle V\rangle_{(0)} + \widehat{X} \langle V\rangle_{(1)} = VX,\quad \forall V\in\C^{\overline{m}}_{p,q,r}\},
\\
\tilde{\Z}^{\overline{km}}_{p,q,r}& :=& 
 \left\lbrace
\begin{array}{lll}
\Z^{\overline{k}}_{p,q,r}\cap\check{\Z}^{\overline{m}}_{p,q,r}&&
\mbox{$k$ is even, $m$ is odd},
\\
\Z^{\overline{k}}_{p,q,r}\cap{\Z}^{\overline{m}}_{p,q,r} = \Z^{\overline{km}}_{p,q,r}, && \mbox{$k=0$, $m=2$},
\\
\check{\Z}^{\overline{k}}_{p,q,r}\cap\check{\Z}^{\overline{m}}_{p,q,r}= \check{\Z}^{\overline{km}}_{p,q,r}, && \mbox{$k=1$, $m=3$};
\end{array}
\right.
\end{eqnarray*}
and write out their explicit forms.

\section{Conclusions}\label{section_conclusions}

In this work, we consider the centralizers and  twisted centralizers in degenerate and non-degenerate Clifford algebras $\C_{p,q,r}$. 
In Theorems \ref{theorem_cc_k_even}, \ref{centralizers_qt}, and \ref{centralizers_qt_ds}, we find an explicit form of the centralizers
and the twisted centralizers 
\begin{eqnarray*}
    \Z^k_{p,q,r},\;\;\Z^{\overline{m}}_{p,q,r},\;\; \Z^{\overline{ml}}_{p,q,r},\qquad \check{\Z}^k_{p,q,r},\;\; \check{\Z}^{\overline{m}}_{p,q,r},\;\; \check{\Z}^{\overline{ml}}_{p,q,r}
\end{eqnarray*}
of the subspaces of fixed grades $\C^{k}_{p,q,r}$, $k=0,1,\ldots,n$, the subspaces $\C^{\overline{m}}_{p,q,r}$, $m=0,1,2,3$, determined by the grade involution and the reversion, and their direct sums $\C^{\overline{ml}}_{p,q,r}$, $m,l=0,1,2,3$. 
In particular, we consider the centralizers and twisted centralizers of the even $\C^{(0)}_{p,q,r}$ and odd $\C^{(1)}_{p,q,r}$ subspaces. 
The relations between $\Z^k_{p,q,r}$ and $\check{\Z}^k_{p,q,r}$ for different $k$ are considered in Remark \ref{CC3CC4}. 
We also consider the relation between their projections $\langle\Z^k_{p,q,r}\rangle_{(1)}$ and $\langle \check{\Z}^k_{p,q,r}\rangle_{(1)}$ onto the odd subspace $\C^{(1)}_{p,q,r}$ in Lemmas \ref{zm_zm1}, \ref{zl_zm} and Remark \ref{rem_zlzm}. 

In the particular cases of the non-degenerate Clifford algebras $\C_{p,q,0}$ and the Grassmann algebras $\C_{0,0,n}$, the considered centralizers and the twisted centralizers have a simpler form than in the general case of arbitrary $\C_{p,q,r}$ (see Remarks \ref{remark_CmCm_pq0} and \ref{remark_00n} respectively).
In the particular case of small $k$, the centralizers $\Z^k_{p,q,r}$ and the twisted centralizers $\check{\Z}^k_{p,q,r}$ have simple form as well and are written out in Remark \ref{cases_we_use} for $k\leq 4$.

In the further research, we are going to use the explicit forms of the centralizers and the twisted centralizers presented in Theorems \ref{theorem_cc_k_even}, \ref{centralizers_qt}, and \ref{centralizers_qt_ds} to define and study several families of Lie
groups in $\C_{p,q,r}$. These groups preserve the subspaces $\C^{\overline{m}}_{p,q,r}$ and their direct sums under the adjoint and twisted adjoint representations. These  Lie groups can be considered as generalizations of Clifford and Lipschitz groups and are important for the theory of spin groups.
We hope that the explicit forms of centralizers and twisted centralizers can be useful for applications of Clifford algebras in physics \cite{phys,hestenes,ce}, computer science, in particular, for neural networks and machine learning \cite{b1,h1,tf,cNN0,cNN}, image processing \cite{cv1,ce}, and in other areas.

\appendix
\section{Twisted centralizers $\tilde{\Z}^{m}_{p,q,r}$, $\tilde{\Z}^{\overline{m}}_{p,q,r}$, and $\tilde{\Z}^{\overline{km}}_{p,q,r}$}\label{sec_tild}
This paper considers the twisted centralizers $\check{\Z}^m_{p,q,r}$, $m=0,\ldots,n$, defined as (\ref{def_chZm}). As recommended by one of the respected reviewers, we consider the other twisted centralizers of the fixed grade subspaces defined as
\begin{eqnarray}
    \tilde{\Z}^m_{p,q,r}:=\{X\in\C_{p,q,r}:\quad X\langle V\rangle_{(0)} + \widehat{X} \langle V\rangle_{(1)} = VX,\quad \forall V\in\C^{m}_{p,q,r}\}.
\end{eqnarray}
This definition corresponds to the twisted adjoint representation $\tilde{\ad}$ (see the formula (\ref{twa22}) in Section \ref{section_examples}). We have 
\begin{eqnarray}
\tilde{\Z}^m_{p,q,r} = 
    \left\lbrace
\begin{array}{lll}
 \Z^m_{p,q,r}, && \mbox{$m$ is even},
\\
\check{\Z}^m_{p,q,r}, && \mbox{$m$ is odd}.
\end{array}
\right.
\end{eqnarray}
In the case $m=0$, we have by (\ref{chCC_0})
\begin{eqnarray}
\tilde{\Z}^{0}_{p,q,r} = \C_{p,q,r}.
\end{eqnarray}
Using Theorem \ref{theorem_cc_k_even}, we get an explicit form (\ref{zt_1})--(\ref{zt_2}) of $\tilde{\Z}^m_{p,q,r}$, $m=1,\ldots,n$.
For an arbitrary even $m$, where $n\geq m\geq 2$, the twisted centralizer has the form
    \begin{eqnarray}\label{zt_1}
    \tilde{\Z}^m_{p,q,r}=\Lambda^{\leq n-m-1}_r\oplus \bigoplus_{k=1\mod{2}}^{m-3} \{\C^{k}_{p,q,0}\Lambda^{\geq n-(m-1)}_r\}\nonumber
    \\
    \oplus\bigoplus_{k=0\mod{2}}^{m-2}\{\C^{k}_{p,q,0}\Lambda^{\geq n-m}_r\} \oplus \C^{n}_{p,q,r}.
    \end{eqnarray}
For an arbitrary odd $m$, where $n\geq m \geq 1$, we have
    \begin{eqnarray}\label{zt_2}
    \tilde{\Z}^m_{p,q,r}=\Lambda_r^{\leq n-m-1}\oplus\!\!\!\!\!\!\bigoplus_{k=1\mod{2}}^{m-2}\!\!\!\!\{\C^{k}_{p,q,0}\Lambda_r^{\geq n-(m-1)}\}\oplus\!\!\!\!\!\!\bigoplus_{k=0\mod{2}}^{m-1}\!\!\!\! \{\C^k_{p,q,0}\Lambda_r^{\geq n-m}\}.
    \end{eqnarray}

Consider the twisted centralizers of the subspaces $\C^{\overline{m}}_{p,q,r}$, $m=0,1,2,3$:
\begin{eqnarray} 
    \tilde{\Z}^{\overline{m}}_{p,q,r}:=\{X\in\C_{p,q,r}:\quad X\langle V\rangle_{(0)} + \widehat{X} \langle V\rangle_{(1)} = VX,\quad \forall V\in\C^{\overline{m}}_{p,q,r}\}.
\end{eqnarray}
Note that
\begin{eqnarray}
\tilde{\Z}^{\overline{m}}_{p,q,r} = 
    \left\lbrace
\begin{array}{lll}
 \Z^{\overline{m}}_{p,q,r}, && \mbox{$m$ is even},
\\
\check{\Z}^{\overline{m}}_{p,q,r}, && \mbox{$m$ is odd}.
\end{array}
\right.
\end{eqnarray}
We have
\begin{eqnarray}
    \tilde{\Z}^{\overline{0}}_{p,q,r} = \Z^4_{p,q,r};\quad \tilde{\Z}^{\overline{k}}_{p,q,r}  = \check{\Z}^k_{p,q,r},\quad k=1,3;\quad \tilde{\Z}^{\overline{2}}_{p,q,r} = \Z^{2}_{p,q,r}.
\end{eqnarray}

Consider the twisted centralizers of the direct sums $\C^{\overline{km}}_{p,q,r}$, $k,m=0,1,2,3$:
\begin{eqnarray}
\tilde{\Z}^{\overline{km}}_{p,q,r} := 
 \left\lbrace
\begin{array}{lll}
\Z^{\overline{k}}_{p,q,r}\cap\check{\Z}^{\overline{m}}_{p,q,r}&&
\mbox{$k$ is even, $m$ is odd},
\\
\Z^{\overline{k}}_{p,q,r}\cap{\Z}^{\overline{m}}_{p,q,r} = \Z^{\overline{km}}_{p,q,r}, && \mbox{$k=0$, $m=2$},
\\
\check{\Z}^{\overline{k}}_{p,q,r}\cap\check{\Z}^{\overline{m}}_{p,q,r}= \check{\Z}^{\overline{km}}_{p,q,r}, && \mbox{$k=1$, $m=3$}.
\end{array}
\right.
\end{eqnarray}
We have
    \begin{eqnarray}
        &\tilde{\Z}^{\overline{02}}_{p,q,r} = {\Z}^{\overline{02}}_{p,q,r}=\Z^2_{p,q,r} =\left\lbrace
    \begin{array}{lll}
\Lambda_r\oplus\C^{n}_{p,q,r},&& r\neq n,
\\
\Lambda_r,&& r=n;
\end{array}
    \right.\
    \\
    &\tilde{\Z}^{\overline{13}}_{p,q,r} = \check{\Z}^{\overline{13}}_{p,q,r}=\tilde{\Z}^{\overline{01}}_{p,q,r}=\tilde{\Z}^{\overline{12}}_{p,q,r}=\Lambda_r,
        \\
        & \tilde{\Z}^{\overline{03}}_{p,q,r} = 
            \left\lbrace
\begin{array}{lll}
        \Lambda_r\oplus \{\C^1_{p,q,0}\Lambda^{\geq n-2}_r\}\oplus \{\C^2_{p,q,0}\Lambda^{\geq n-3}_r\},&& r\neq n,
        \\
        \Lambda_r,&&r=n;
        \end{array}
\right.
\\
& \tilde{\Z}^{\overline{23}}_{p,q,r} = 
\left\lbrace
\begin{array}{lll}
\Lambda_r\oplus \{\C^1_{p,q,0}\Lambda^{n-1}_r\}\oplus\{\C^2_{p,q,0}\Lambda^{n-2}_r\}, && r\neq n,
\\
\Lambda_r, && r=n.
\end{array}
\right.
    \end{eqnarray}

\section*{Acknowledgements}

The results of this paper were reported at the 13th International Conference on Clifford Algebras and Their Applications in Mathematical Physics, Holon, Israel, June 2023. The authors are grateful to the organizers and the participants of this conference for fruitful discussions.

The authors are grateful to the anonymous reviewers for valuable comments on how to improve the presentation.

The publication was prepared within the framework of the Academic Fund Program at HSE University (grant 24-00-001 Clifford algebras and applications).

\medskip

\noindent{\bf Data availability} Data sharing not applicable to this article as no datasets were generated or analyzed during the current study.

\medskip

\noindent{\bf Declarations}\\
\noindent{\bf Conflict of interest} The authors declare that they have no conflict of interest.

\bibliographystyle{spmpsci}

\end{document}